\documentclass[11pt,a4paper]{article}
\pdfoutput=1

\usepackage[utf8]{inputenc}

\usepackage{fullpage} %

\usepackage{tikz}
\usetikzlibrary{calc,math,patterns,shapes,backgrounds}
\tikzset{cross/.style={cross out, draw=black, minimum size=2*(#1-\pgflinewidth), inner sep=0pt, outer sep=0pt},
cross/.default={1pt}}

\newcommand{\diamondleft}[4]{[  ]
	\pgfmathsetmacro\n{(#4 + #3 - #1 - #2) / 2};
	\draw (#1,#2) -- (#3 -\n , #4 - \n) -- (#3,#4) -- (#1+\n ,#2 +\n) -- (#1,#2);
}
\newcommand{\diamondup}[4]{[  ]
	\pgfmathsetmacro\n{(#2 + #3 - #1 - #4) / 2};
	\pgfmathsetmacro\m{(- #3 + \n + #1)};
	\draw (#1,#2) -- (#3 -\n , #4 + \n) -- (#3,#4) -- (#1+\n ,#2-\n) -- (#1,#2);
}
\newcommand{\diamondleftlabel}[6]{[  ]
	\pgfmathsetmacro\n{(#4 + #3 - #1 - #2) / 2};
	\pgfmathsetmacro\lx{(#3 -#1)/2 + #1};
	\pgfmathsetmacro\ly{(#2 + \n - #4 +\n)/2 + #4 - \n};
	\draw[#5] (#1,#2) -- (#3 -\n , #4 - \n) -- (#3,#4) -- (#1+\n ,#2 +\n) -- (#1,#2);
	\node at (\lx,\ly) {#6};
}
\newcommand{\diamonduplabel}[6]{[  ]
	\pgfmathsetmacro\n{(#2 + #3 - #1 - #4) / 2};
	\pgfmathsetmacro\lx{(#3 -\n #1+\n)/2 + #3 -\n };
	\pgfmathsetmacro\ly{(#2 - #4)/2 + #4};
	\draw[#5] (#1,#2) -- (#3 -\n , #4 + \n) -- (#3,#4) -- (#1+\n ,#2-\n) -- (#1,#2);
	\node at (\lx,\ly) {#6};
}
\tikzset{2Dpoint/.style ={circle,inner sep=0pt,minimum size=3pt,fill=black}}
\tikzset{bold2Dpoint/.style ={inner sep=0pt,minimum size=4pt,fill=black}}
\tikzset{vertex/.style ={circle,inner sep=0pt,minimum size=4pt}}
\tikzset{cross/.pic = {
    \draw (-#1,0) -- (#1,0);
    \draw (0,-#1) -- (0, #1);
    }
}
\tikzstyle{path} = [color=black,opacity=.15,line cap=round, line join=round, line width=6pt]

\newcommand{\drawCoordOrigin}[2]{
		\draw[->,thick] (#1,#2) -- (#1 + 6,#2);
		\draw[->,thick] (#1,#2) -- (#1,#2 - 6);
		\draw[thick] (#1,#2) node[label=above left:{0}] {} pic {cross=2pt};
		\foreach \i in {1,...,5}{
			\draw[thick] (#1,#2 - \i) node[label=left:{\i}] {} pic {cross=2pt};
			\draw[thick] (#1 + \i,#2) node[label=above:{\i}] {} pic {cross=2pt};
		}
}

\usepackage{thmtools}
\usepackage{thm-restate}

\definecolor{italyGreen}{RGB}{0, 146, 70}
\definecolor{italyRed}{RGB}{206, 43, 55}

\usepackage{amssymb}
\usepackage{amsthm}
\usepackage{mathtools}
\usepackage{xparse}
\usepackage{dsfont}
\usepackage{booktabs}
\usepackage[numbers,sort]{natbib}

\usepackage{caption}
\usepackage{enumerate}
\usepackage[disable]{todonotes}
\usepackage{tabularx}
\usepackage{multirow}
\usepackage[d]{esvect}
\usepackage{etoolbox}

\usepackage[ruled,vlined,linesnumbered]{algorithm2e}
\DontPrintSemicolon
\SetKwInOut{Input}{Input}
\SetKwInOut{Output}{Output}
\SetKw{Break}{break}
\SetKwFunction{solve}{solve}
\SetKwFunction{towerGuessing}{towerGuessing}
\SetKwFunction{dynProg}{dynamicProgram}
\SetKwFunction{phasetwo}{Prep}
\SetKwProg{myproc}{function}{}{}

\usepackage{caption}
\usepackage[shortlabels]{enumitem}

\usepackage[d]{esvect}

\usepackage{sidecap}
\sidecaptionvpos{figure}{t}

\usepackage[pagebackref,menucolor=orange!40!black,filecolor=magenta!40!black,urlcolor=blue!40!black,linkcolor=red!40!black,citecolor=green!40!black,colorlinks]{hyperref}
\usepackage{doi}
\usepackage[sort&compress,nameinlink,noabbrev,capitalize]{cleveref}

\DeclarePairedDelimiterX{\set}[1]{\{ }{ \} }{\setargs{#1}}
\NewDocumentCommand{\setargs}{>{\SplitArgument{1}{;}}m}
{\setargsaux#1}
\NewDocumentCommand{\setargsaux}{mm}
{\IfNoValueTF{#2}{#1} {#1\mid\mathopen{}#2}}%

\DeclarePairedDelimiterX{\abs}[1]{\lvert}{\rvert}{#1}
\DeclarePairedDelimiterX{\ceil}[1]{\lceil}{\rceil}{#1}
\DeclarePairedDelimiterX{\floor}[1]{\lfloor}{\rfloor}{#1}
\DeclarePairedDelimiterX{\norm}[1]{\lVert}{\rVert}{#1}
\DeclarePairedDelimiterX{\infnorm}[1]{\lVert}{\rVert_\infty}{#1}

\theoremstyle{plain}

\newtheorem{lemma}{Lemma}
\newtheorem{proposition}[lemma]{Proposition}

\newtheorem{observation}[lemma]{Observation}

\theoremstyle{definition}
\newtheorem{definition}[lemma]{Definition}

\crefname{observation}{Observation}{Observations}
\crefname{claim}{Claim}{Claims}
\crefalias{AlgoLine}{line}

\makeatletter
\let\cref@old@stepcounter\stepcounter
\def\stepcounter#1{%
  \cref@old@stepcounter{#1}%
  \cref@constructprefix{#1}{\cref@result}%
  \@ifundefined{cref@#1@alias}%
    {\def\@tempa{#1}}%
    {\def\@tempa{\csname cref@#1@alias\endcsname}}%
  \protected@edef\cref@currentlabel{%
    [\@tempa][\arabic{#1}][\cref@result]%
    \csname p@#1\endcsname\csname the#1\endcsname}}
\makeatother

\newcommand{\DSP}[1]{$#1$-DSP}
\newcommand{\kDSP}{\DSP{k}}

\newcommand{\DPD}{\textsc{$p$-Disjoint Paths on DAGs}}

\newcommand{\NN}{\mathds{N}}

\newcommand{\RR}{\mathds{R}}

\newcommand{\yes}{\emph{yes}}

\newcommand{\varmarks}{\textnormal{marks}}
\newcommand{\varlabels}{\textnormal{labels}}

\newcommand*{\defeq}{:=}

\newcommand{\R}{\mathds{R}}

\newcommand{\calP}{\mathcal{P}}
\newcommand{\cross}{\mathcal{M}}
\newcommand{\crossing}{\mathcal{C}}

\DeclareMathOperator{\true}{true}
\DeclareMathOperator{\false}{false}
\DeclareMathOperator*{\argmin}{arg\,min}
\DeclareMathOperator*{\argmax}{arg\,max}

\DeclareMathOperator{\dist}{dist}
\newcommand{\dis}{\operatorname{dist}}

\DeclareMathOperator{\start}{start}
\DeclareMathOperator{\en}{end}

\DeclareMathOperator{\setp}{set}
\DeclareMathOperator{\iin}{in}
\DeclareMathOperator{\oout}{out}

\DeclareMathOperator{\segmentEnds}{\mathcal{T}}
\DeclareMathOperator{\marbles}{Ends}

\newcommand{\patharea}[2]{#1 \mathbin{\diamond} #2}

\DeclareRobustCommand{\pos}[1]{{\vv{#1}}}

\newcommand{\wilog}{without loss of generality}
\newcommand{\Wilog}{Without loss of generality}
\newcommand{\ie}{i.\,e.,}
\newcommand{\eg}{e.\,g.\;}

\newcommand{\mpath}{marble path}

\newcommand{\probDef}[4]{
\begin{center}   
	\fbox{~\begin{minipage}{.95\textwidth}
		\vspace{2pt} 

		\noindent
		\normalsize\textsc{#1}
		
		\vspace{4pt}
		\setlength{\tabcolsep}{3pt}
		\renewcommand{\arraystretch}{1.0}
		\begin{tabularx}{\textwidth}{@{}lX@{}}
			\normalsize\textbf{Input:} 	& \normalsize#2 \\
			\normalsize\textbf{#4:} 	& \normalsize#3
		\end{tabularx}
	\end{minipage}}
\end{center}
}

\newcommand{\optProb}[3]{\probDef{#1}{#2}{#3}{Task}}
\newcommand{\decProb}[3]{\probDef{#1}{#2}{#3}{Question}}

\usepackage{authblk}

\tikzset{nodesmall/.style args={#1}{draw,circle,label=above left:{#1}}}
\tikzstyle{small}=[inner sep=2pt]
\tikzset{marbler/.style ={circle,inner sep=0pt,minimum size=5pt,shade,top color=white, bottom color=red}}
\tikzset{marbleb/.style ={circle,inner sep=0pt,minimum size=5pt,shade,top color=white, bottom color=blue}}

\newcommand{\tuaddress}{Technische Universität Berlin, Faculty IV, Algorithmics and Computational Complexity, Germany}

\title{Using a geometric lens to find $k$ disjoint shortest paths\footnote{This work was started at the research retreat of the group \emph{Algorithmics and Computational Complexity} in September 2019, held in Schloss Neuhausen (Prignitz).}}

\author{Matthias~Bentert}
\author{André~Nichterlein}
\author{Malte~Renken\thanks{Supported by the DFG project MATE (NI 369/17).}}
\author{Philipp~Zschoche}
\affil{\tuaddress}
\affil{\{matthias.bentert, andre.nichterlein, m.renken, zschoche\}@tu-berlin.de}

\date{}

\begin{document}
\maketitle

\begin{abstract}
Given an undirected $n$-vertex graph and $k$~pairs of terminal vertices~$(s_1,t_1), \ldots, (s_k,t_k)$, the \textsc{$k$-Disjoint Shortest Paths} (\kDSP) problem asks
whether there are $k$ pairwise vertex-disjoint paths~$P_1, \allowbreak\ldots, P_k$ such that~$P_i$ is a shortest~$s_i$-$t_i$-path for each~$i \in [k]$.
Recently, \citeauthor{Loc21} [SODA~'21] provided an algorithm that solves \kDSP{} in~$n^{O\left(k^{5^k}\right)}$ time, answering a~20-year old question about the computational complexity of \kDSP{} for constant~$k$.
On the one hand, we present an improved $n^{O(k!k)}$-time algorithm based on a novel geometric view on this problem.
For the special case~$k=2$ on $m$-edge graphs, we show that the running time can be further reduced to~$O(nm)$ by small modifications of the algorithm and a refined analysis.
On the other hand, we show that \kDSP{} is W[1]-hard with respect to~$k$, showing that the dependency of the degree of the polynomial running time on the parameter~$k$ is presumably unavoidable.
\end{abstract}

\section{Introduction}
The \textsc{$k$-Disjoint Paths} problem is a fundamental and well-studied combinatorial problem.
Given an undirected graph~$G$ on $n$ vertices and~$k$ terminal pairs $(s_i, t_i)_{i \in [k]}$, the question is whether there are pairwise disjoint%
\footnote{We only consider the vertex-disjoint setting to which the edge-disjoint version can be reduced to. We will also see later that the hardness results also hold for the edge-disjoint version.} 
$s_i$-$t_i$-paths~$P_i$ for each~$i \in [k]$.
The problem was shown to be NP-hard by~\citet{Kar75} when~$k$ is part of the input.
On the positive side, \citet{RS95} provided an algorithm running in~$O(n^3)$ time for any constant~$k$.
Later, \mbox{\citet{KKR12}} improved the running time to~$O(n^2)$, again for fixed~$k$.
On directed graphs, in contrast, the problem is NP-hard even for~$k=2$~\cite{FHW80}.
However, on directed acyclic graphs, the problem becomes again polynomial-time solvable for constant~$k$~\cite{FHW80} and linear-time solvable for~$k=2$~\cite{Tho12}.

Focusing on the undirected case, we study the problem variant where all paths in the solution have to be shortest paths. 
This variant was introduced by \citet{Eil98}.

\optProb{$k$ Disjoint Shortest Paths (\kDSP)}
{A graph~$G=(V,E)$ and $k \in \NN$~pairs~$(s_i,t_i)_{i \in [k]}$ of vertices.}
{Find $k$ disjoint paths $P_i$ such that $P_i$ is a shortest~$s_i$-$t_i$-path for each~$i \in [k]$.}

Throughout the paper, we assume that the input graph is connected.
\citet{Eil98} showed the NP-hardness of \kDSP{} when~$k$ is part of the input.
Moreover, \citeauthor{Eil98} provided a dynamic-programming based~$O(n^8)$-time algorithm for \DSP{2}.
This algorithm for \DSP{2} works also for positive edge lengths.
Recently, \mbox{\citet{GKW19}} and \citet{KS19} independently extended this result by providing polynomial-time algorithms for the case that the edge lengths are non-negative.
Very recently, \citet{Akh20} presented an algorithm solving~\DSP{2} in $O(n^6)$ time for unweighted graphs and in~$O(n^7)$ for positive edge lengths. 
As for directed graphs, \citet{BK17} provided a polynomial-time algorithm for strictly positive edge length.
Note that allowing zero-length edges generalizes \textsc{$2$-Disjoint Path} on directed graphs, which is NP-hard~\cite{FHW80}. 
Extending the problem to finding two disjoint $s_i$-$t_i$-paths of minimal total length (in undirected graphs), \citet{BH19} provided a randomized algorithm running in~$O(n^{11})$ time.

The existing algorithms for \DSP{2} are based on dynamic programming with tedious case distinctions.
We provide a new algorithm using a simple and elegant geometric perspective: %
\begin{restatable}
{theorem}{dsptwoThm}
\label{prop:2dsp}
\DSP{2} can be solved in $O(nm)$ time.
\end{restatable}

Whether or not \kDSP{} for constant~$k \ge 3$ is polynomial-time solvable was first asked by \citet{Eil98} over 20 years ago and was posed again as a research challenge in 2019~\cite[open problem 4.6]{FMSZ19}.
Recently, \citet{Loc21} settled this long standing open question by showing that \kDSP{} can be solved in~$n^{O(k^{5^k})}$ time.
Using our geometric approach, we provide an improved~$O(k \cdot n^{16k \cdot k! +k + 1})$-time algorithm. 
\begin{restatable}
		{theorem}{algorithmThm}
	\label{thm:main-thm}
	\kDSP{} can be solved in~$O(k \cdot n^{16k \cdot k! +k + 1})$ time.
\end{restatable}

We describe the basic idea of our algorithms and the new geometric tools in \cref{sec:poly-time}. %
In \cref{sec:2D-geometry}, we formalize these geometric tools for two paths and prove \cref{prop:2dsp}.
In \cref{sec:towers}, we lift these arguments to $k > 2$ paths.
In \cref{sec:the-algorithm}, we present our algorithm for \cref{thm:main-thm} and prove its correctness.

Finally, in \cref{sec:eth}, 
we show that \kDSP{} is W[1]-hard with respect to~$k$.
Hence, under standard assumptions from parameterized complexity, there is no algorithm with running time~$f(k) n^{O(1)}$ for any function~$f$.
Thus, polynomial-time algorithms where~$k$ does not appear in the exponent (as the~$O(n^2)$-time algorithm for \textsc{$k$-Disjoint Path} for any constant~$k$~\cite{RS95,KKR12}) are unlikely to exist for \kDSP.
Furthermore, under the Exponential-Time Hypothesis (ETH), we show that there is no algorithm with running time~$f(k)\cdot n^{o(k)}$ for any computable function~$f$.

\begin{restatable}{proposition}{hardnessThm}
	\label{thm:eth}
	\kDSP{} is W[1]-hard with respect to~$k$.
	Moreover, assuming the ETH, there is no $f(k)\cdot n^{o(k)}$-time algorithm for \kDSP.
\end{restatable}

\subparagraph*{Preliminaries.} %
We set~$\NN := \{0,1,2,\ldots,\}$ and $[n] := \{1,2,\dots,n\}$.
We always denote by~$G = (V,E)$ a graph (undirected and connected unless explicitly stated otherwise) and by~$n$ and~$m$ the number of vertices and edges in~$G$, respectively.
A \emph{path} of length $\ell \geq 0$ in a graph~$G$ is a sequence of distinct vertices $v_0v_1\dots v_\ell$ such that each pair $v_{i-1}, v_i$ is connected by an edge in~$G$.
When no ambiguity arises, we do not distinguish between a path and its set of vertices.
The first and last vertex $v_0$ and $v_\ell$ of a path~$P$ are called the \emph{end vertices} or \emph{ends} of $P$ and are denoted by $s_P$ and $t_P$.
We also say that $P$ is a path \emph{from} $v_0$ \emph{to} $v_\ell$, a path \emph{between}~$v_0$ and~$v_\ell$, or a~$v_0$-$v_\ell$-path.
For~$v,w \in P$, we denote by~$P[v,w]$ the subpath of~$P$ with end vertices~$v$ and~$w$.
For two vertices~$v,w$, we denote the length of a shortest~$v$-$w$-path in~$G$ by~$\dist_G(v,w)$ or $\dist(v,w)$ if the graph~$G$ is clear from the context.
If for all $i \in [k]$ there is a path $P_i$ that is a shortest~$s_i$-$t_i$-path and disjoint with $P_j$ for all $j \in[k]\setminus\{i\}$, then we say that the paths~$(P_i)_{i \in [k]}$ are a solution for an instance $(G,(s_i,t_i)_{i\in[k]})$ of \kDSP.

\section{The Key Concepts behind our Polynomial-Time Algorithm} \label{sec:poly-time}

In this section, we describe our approach to solve \kDSP{} in polynomial-time for any fixed~$k$.
As a warm-up, we start with sketching an algorithm for \DSP{2} based on the same approach. 

\subparagraph*{Solving \DSP{2} in the plane.}

Before describing the algorithm, we show the central geometric idea behind it.  
Recall that we want to find two shortest paths~$P_1$ and~$P_2$ from~$s_1$ to~$t_1$ and from~$s_2$ to~$t_2$, respectively.
We now arrange the vertices on a 2-dimensional grid where the first coordinate of each vertex is the distance to~$s_1$ and the second coordinate the distance to~$s_2$; see left side of \cref{fig:diamond-introduction} for an example graph with the corresponding coordinates and the right side for an arrangement of the vertices in a grid with a continuous drawing of the paths (drawing straight lines between points occurring in the paths).
Clearly, with two breadth-first searches from~$s_1$ and~$s_2$, we can compute the coordinates of all vertices in linear time.
Note that there might be multiple vertices with the same coordinates.
However, at most one vertex per coordinate can be part of a shortest $s_1$-$t_1$- or $s_2$-$t_2$-path.
\begin{figure}[t]
	\centering
	\small

	\begin{tikzpicture}[scale=.6]
		\begin{scope}[xshift=-2.5cm,yshift=-6cm,yscale=1.25]
			\node[circle,inner sep=0pt,minimum size=3pt,fill=black,label=left:{\textbf{\boldmath{$s_1$}}},label={[shift={(0,-.1)}]{$\binom{0}{3}$}}] (s1) at (0,2){};
			\node[circle,inner sep=0pt,minimum size=3pt,fill=black,label={[shift={(-.15,-.1)}]{$\binom{1}{2}$}}] (a1) at (2,2){};
			\node[circle,inner sep=0pt,minimum size=3pt,fill=black,label={[shift={(0,-.1)}]{$\binom{2}{2}$}}] (a2) at (3,2.75){};
			\node[circle,inner sep=0pt,minimum size=3pt,fill=black,label={[shift={(.15,-.1)}]{$\binom{2}{1}$}}] (a3) at (4,2){};
			\node[circle,inner sep=0pt,minimum size=3pt,fill=black,label=right:{\textbf{\boldmath{$s_2$}}},label={[shift={(.1,-.1)}]{$\binom{3}{0}$}}] (s2) at (6,2){};
			
			\node[circle,inner sep=0pt,minimum size=3pt,fill=black,label=below:{\textbf{\boldmath{$t_1$}}},label={[shift={(-.1,-.1)}]{$\binom{5}{4}$}}] (t1) at (-1.5,.5){};
			\node[circle,inner sep=0pt,minimum size=3pt,fill=black,label={[shift={(-.2,-.1)}]{$\binom{4}{3}$}}] (b0) at (0,.5){};
			\node[circle,inner sep=0pt,minimum size=3pt,fill=black,label={[shift={(-.2,-.1)}]{$\binom{3}{3}$}}] (b1) at (2,.5){};
			\node[circle,inner sep=0pt,minimum size=3pt,fill=black,label=below:{$\binom{2}{2}$}] (b2) at (3,1.25){};
			\node[circle,inner sep=0pt,minimum size=3pt,fill=black,label={[shift={(.2,-.1)}]{$\binom{3}{3}$}}] (b3) at (4,.5){};
			\node[circle,inner sep=0pt,minimum size=3pt,fill=black,label={[shift={(0,-.1)}]{$\binom{4}{4}$}}] (b4) at (6,.5){};
			\node[circle,inner sep=0pt,minimum size=3pt,fill=black,label=below:{\textbf{\boldmath{$t_2$}}},label={[shift={(0,-.1)}]{$\binom{5}{5}$}}] (t2) at (7.5,.5){};	
			
			\node[circle,inner sep=0pt,minimum size=3pt,fill=black,label={[shift={(0,-.1)}]{$\binom{5}{2}$}}] (c1) at (1.5,4.25){};
			\node[circle,inner sep=0pt,minimum size=3pt,fill=black,label={[shift={(0,-.1)}]{$\binom{4}{1}$}}] (c2) at (4.5,4.25){};	
			\foreach \x/\y in {s1/a1, a1/a2, a2/a3, a1/a2, a3/s2, a1/b2, a2/b1, a1/a3, a2/b3, b2/a3, t1/b0, b0/b1, b1/b2,b2/b3,b3/b4,b4/t2, b0/c1, c1/c2, c2/s2}{
				\draw (\x) -- (\y);
			} 
			\begin{pgfonlayer}{background}
				\draw[path,draw=red] (s1.center) -- (a1.center) -- (a2.center) -- (b1.center) -- (t1.center);
				\draw[path,draw=blue,dash pattern=on 4pt off 8pt] (s2.center) -- (a3.center) -- (b2.center) -- (b3.center) -- (t2.center);
			\end{pgfonlayer}
		\end{scope}     
		\small
		\begin{scope}[xshift=7cm]
			\foreach \x in {0,1,...,6}{
				\foreach \y in {0,-1,...,-6}{
					\node[cross, gray] at (\x,\y){};
				}
			}
			\drawCoordOrigin{0}{0}
		
			\foreach \x/\y in {0/-3,1/-2,2/-2,2/-1,3/0,4/-1, 5/-2,3/-3,5/-4,4/-3, 5/-5, 4/-4}{
				\node[2Dpoint] at (\x,\y){};
			}

			\filldraw (0,-3) circle node[right] {\textbf{\boldmath{$s_1$}}} --
					(2,-1) circle node[right] {} --
					(5,-4) circle node[right] {\textbf{\boldmath{$t_1$}}} --
					(3,-6) circle node[right] {} -- (0,-3);
			\filldraw(1.5,-1.5) circle node[left] {} --
					(3,0) circle node[below,yshift=-3pt] {\textbf{\boldmath{$s_2$}}} --
					(6.5,-3.5) circle node[right] {} --
					(5,-5) circle node[below] {\textbf{\boldmath{$t_2$}}} -- (1.5,-1.5);
			\diamonduplabel{3}{0}{5}{-5}{fill=blue, opacity=.2}{};
			\diamondleftlabel{0}{-3}{5}{-4}{fill=red, opacity=.2}{};
			\draw[red, ultra thick] (0,-3) --(1,-2) --(2,-2)--(3,-3)--(4,-3)--(5,-4);
			\draw[blue,very thick, dashed] (3,0) --(2,-1) --(2,-2)--(5,-5);
		\end{scope} 
	\end{tikzpicture} 
	\hspace{-0.4cm}
	\caption{
		\emph{Left side:} A graph with distinguished vertices~$s_1,s_2,t_1,t_2$. The vertex coordinates are written next to the vertices and two shortest paths are highlighted.
		\emph{Right side:} 
		The 2D-arrangement of the vertices. 
		The red and blue rectangles spanned by~$s_1$ and~$t_1$ and by~$s_2$ and~$t_2$ contain the shortest $s_1$-$t_1$-path (solid red) and~$s_2$-$t_2$-path (dashed blue).
	}
	\label{fig:diamond-introduction}
	\vspace{-0.4cm}
\end{figure}
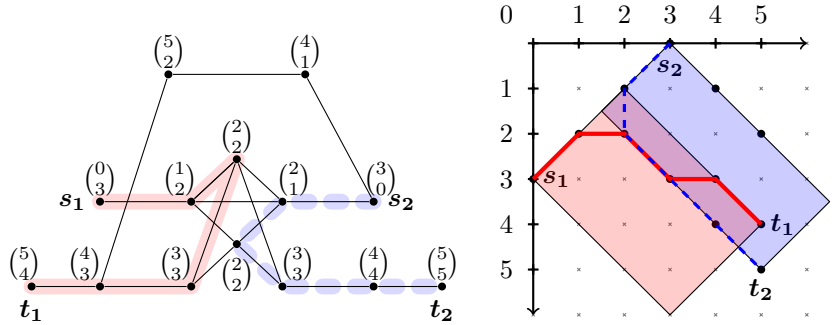

This arrangement of the vertices allows the following simple geometric observation: The drawing of each shortest $s_1$-$t_1$-path has to be within a rectangle with angles of $45^\circ$ to the coordinate axes and with corner points~$s_1$ and~$t_1$ (see red and blue rectangles in the right side of \cref{fig:diamond-introduction}). 
As a consequence, shortest paths that have to stay within two disjoint such rectangles cannot intersect.
We use this argument extensively in the subsequently sketched algorithm.%

We first assume that the drawings of~$P_1$ and~$P_2$ cross as displayed in the right side of \cref{fig:diamond-introduction} (the other case is significantly easier to deal with and we will explain it when describing the full algorithm for \kDSP).
Therein, we distinguish whether the intersection of the drawings of~$P_1$ and~$P_2$ contain a point with integer coordinates or not, that is, our algorithm tries to find solutions for both cases.

If the intersection does not have a point with integer coordinates, then it is easy to see that the intersection of the drawing of~$P_1$ and~$P_2$ has to be a single point~$p$ (with non-integer coordinates).
We guess\footnote{Whenever we pretend to guess something, the algorithm actually exhaustively tests all possible choices.} the four coordinate pairs~$(x,y),(x,y+1),(x+1,y),$ and~$(x+1,y+1)$ with~$x,y \in \NN$ surrounding the intersection point of the drawings of~$P_1$ and~$P_2$.
Note that this can be done in~$O(n)$ time by guessing the vertex on the coordinate pair~$(x,y)$ and guessing whether it belongs to~$P_1$ or~$P_2$.
The goal is now to turn the input graph~$G$ into a directed acyclic graph~$D$ such that each shortest $s_i$-$t_i$-path in~$G$ corresponds to an $s_i$-$t_i$-path in~$D$.
To this end, we partition the grid into four areas~$A_1, B_1, A_2, B_2$ (each area is defined by one of the guessed points and the closer endpoint of the path going through the point) and orient each edge according to the area it lies in (see left side of \cref{fig:to-dag} for an illustration).
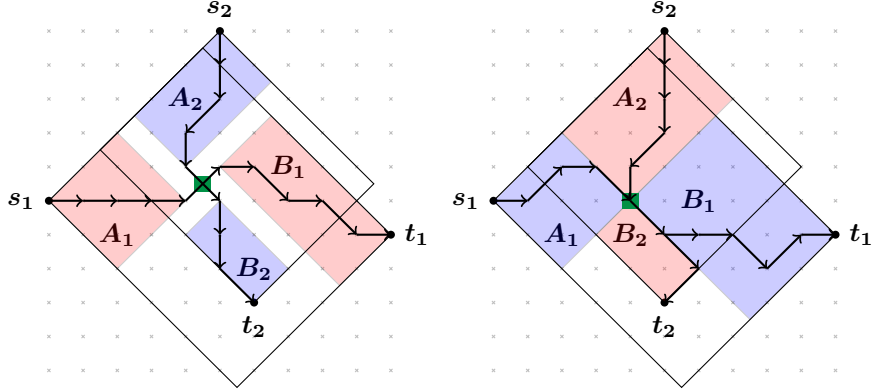
\begin{figure}[t]
	\centering
	\hspace{-0.4cm}
	\begin{tikzpicture}[scale=.45]
	\small
	\begin{scope}[xshift=0cm,rotate=90]
		\foreach \x in {0,1,...,10}{
			\foreach \y in {0,1,...,10}{
				\node[cross, gray!50] at (\x,\y){};
			}
		}
		\node[2Dpoint,label=below:\textbf{\boldmath{$t_2$}}] at (2,5) {};
		\node[2Dpoint,label=above:\textbf{\boldmath{$s_2$}}] at (10,5) {};
		\node[2Dpoint,label=right:\textbf{\boldmath{$t_1$}}] at (4,0) {};
		\node[2Dpoint,label=left:\textbf{\boldmath{$s_1$}}] at (5,10) {};
		
		\node at (4,6) {\textbf{\boldmath{$B_2$}}};
		\node at (8,6) {\textbf{\boldmath{$A_2$}}};

		\node at (4,8) {\textbf{\boldmath{$A_1$}}};
		\node at (5,4) {\textbf{\boldmath{$B_1$}}};

		\diamondleftlabel{2}{5}{5}{6}{fill=red,  opacity=.2}{}; %
		\diamondleftlabel{5}{6}{10}{5}{fill=red,  opacity=.2}{}; %
		\diamonduplabel{5}{10}{5}{6}{fill=blue,  opacity=.2}{};; %
		\diamonduplabel{5}{6}{4}{0}{fill=blue, opacity=.2}{}; %

		\node[bold2Dpoint,minimum size = 6pt,italyGreen] at (5,6) {};
	
		\diamondleft{4}{0}{5}{10};
		\diamondleft{2}{5}{10}{5};
		
		\draw[thick,<-] (2,5) -- (3,4);
		\draw[thick,<-] (3,4) -- (4,5);

		\draw[thick,->]  (5,6) -- (4,5) ;

		\draw[thick,->]  (10,5) -- (9,5);
		\draw[thick,->]  (9,5) -- (8,5);
		\draw[thick,->]  (8,5) -- (7,5);
		\draw[thick,->]  (7,5) -- (6,6);
		\draw[thick,->]  (6,6) -- (5,6); 

		\draw[thick,<-] (4,0) -- (4, 1);
		\draw[thick,<-] (4,1) -- (3, 2);
		\draw[thick,<-] (3, 2) -- (4, 3);
		\draw[thick,<-] (4, 3)-- (4,4);
		\draw[thick,<-] (4,4) -- (4,5) ;

		\draw[thick,<-] (5,6) -- (6,7);
		\draw[thick,<-] (6,7) -- (6, 8);
		\draw[thick,<-] (6, 8) -- (5, 9);
		\draw[thick,<-] (5, 9) --  (5, 10); 
	\end{scope}
	\begin{scope}[xshift=-23cm]
		\foreach \x in {0,1,...,10}{
			\foreach \y in {0,1,...,10}{
				\node[cross, gray!50] at (\x,\y){};
			}
		}
		\node[2Dpoint,label=left:\textbf{\boldmath{$s_1$}}] at (0,5) {};
		\node[2Dpoint,label=right:\textbf{\boldmath{$t_1$}}] at (10,4) {};
		\node[2Dpoint,label=below:\textbf{\boldmath{$t_2$}}] at (6,2) {};
		\node[2Dpoint,label=above:\textbf{\boldmath{$s_2$}}] at (5,10) {};
		\node at (2,4) {\textbf{\boldmath{$A_1$}}};
		\node at (7,6) {\textbf{\boldmath{$B_1$}}};
		\node at (4,8) {\textbf{\boldmath{$A_2$}}};
		\node at (6,3) {\textbf{\boldmath{$B_2$}}};

		\diamondleftlabel{0}{5}{4}{5}{fill=red,  opacity=.2}{};
		\diamondleftlabel{5}{6}{10}{4}{fill=red,  opacity=.2}{};
		\diamonduplabel{5}{10}{4}{6}{fill=blue,  opacity=.2}{};;
		\diamonduplabel{5}{5}{6}{2}{fill=blue, opacity=.2}{};
		
		\node[bold2Dpoint,minimum size = 6pt,italyGreen] at (4.5,5.5) {};
		\diamondleft{0}{5}{10}{4};
		\diamondleft{5}{10}{6}{2};
		
		\draw[thick,->] (4,5) -- (5,6) ;
		\draw[thick,->] (4,6) --(5,5); 

		\draw[thick,->] (0,5) -- (1,5) ;
		\draw[thick,->] (1,5) -- (2,5) ;
		\draw[thick,->] (2,5) -- (3,5) ;
		\draw[thick,->] (3,5) -- (4,5) ;

		\draw[thick,<-] (6,2) --(5,3); 
		\draw[thick,<-] (5,3) --(5,4); 
		\draw[thick,<-] (5,4) --(5,5); 

		\draw[thick,->] (5,10) --(5,9);
		\draw[thick,->] (5,9) --(5,8);
		\draw[thick,->] (5,8) -- (4,7);
		\draw[thick,->] (4,7) -- (4,6); 
		\draw[thick,<-] (10,4) --(9,4);
		\draw[thick,<-] (9,4) --(8,5);
		\draw[thick,<-] (8,5) --(7,5);
		\draw[thick,<-] (7,5) -- (6,6);
		\draw[thick,<-] (6,6) -- (5,6); 
	\end{scope}
	\end{tikzpicture}
	\caption{ An illustration of the directed acyclic graph used for \cref{prop:2dsp}. 
			The first point~$p$ in the intersection of the drawing of~$P_1$ and~$P_2$ is marked by a green square.
			\emph{Left side:} The straight-line drawings of $P_1$ and~$P_2$ only intersect in a non-integer point.
			\emph{Right side:} Both~$P_1$ and~$P_2$ use vertices with same coordinates.
		}
	\label{fig:to-dag}
\end{figure}
An edge~$\{v,w\}$ in the area~$A_i$ or~$B_i$, $i \in [2]$, is oriented towards the vertex~$w$ with the larger~$i$ coordinate.
Edges between~$A_i$ and~$B_i$ are oriented towards the vertex in~$B_i$.
All remaining (unoriented) edges are removed.
Note that this results in a directed acyclic graph (DAG).
Furthermore, a shortest $s_i$-$t_i$-path in~$G$ that goes through the guessed intersection induces an $s_i$-$t_i$-path in~$D$ and each $s_i$-$t_i$-path in~$D$ is a shortest path in~$G$ because it is strictly monotone increasing in the $i$-coordinate and all strictly monotone increasing paths have the same length as each path contains one vertex for each integer $i$-coordinate between the~$i$-coordinates of~$s_i$ and~$t_i$.
Observe that in this case the two paths cannot intersect as~$P_1$ can only reach vertices with coordinates in~$A_1$ and~$B_1$ and~$P_2$ can only use vertices with coordinates in~$A_2$ and~$B_2$.
Hence, one can find~$P_1$ and~$P_2$ in linear time.
Altogether, this gives a running time of~$O(nm)$ in this case.

Assume now that there is a point with integer coordinates in the intersection of~$P_1$ and~$P_2$; this case requires more work.
We assume that if there are two points $(x_1,y_1)$ and $(x_2,y_2)$ in the intersection of the drawings of $P_1$ and~$P_2$, then we have $x_1 < x_2 \Leftrightarrow y_1 < y_2$.
If this is not the case, then repeat the algorithm below with swapped $s_2$ and $t_2$.
We guess the first point~$p$ in the intersection (note that it has integer coordinates). 
This can be done in $O(n)$ time by guessing a vertex on~$p$.
Now we arrange the areas slightly different.
The areas are defined by~$p$ and one coordinate of~$s_1, t_1, s_2, t_2$; see the right side of \cref{fig:to-dag} for an illustration.
Edges in the area~$A_i \setminus B_j$ or~$B_i \setminus A_j$ are oriented towards the vertex with the larger~$i$-coordinate.
Note that edges on the line in~$A_i \cap B_j, i \neq j,$ could either be used by~$P_1$ or~$P_2$ (but not both), meaning that we have to direct the edges towards the vertex with either the larger $1$- or $2$-coordinate.
Since there are only two possibilities for orienting the edges in~$A_1 \cap B_2$, there are only four different possibilities to orient the edges on~$A_1 \cap B_2$ and~$A_2 \cap B_1$.
We try all four possible orientations and if at least one of them yields a solution, then we know that there is a solution.
All other (unorientated) edges are removed---a shortest $s_i$-$t_i$-path cannot use it.
Note that this results for each of the four described cases in a DAG.
Furthermore, again a shortest $s_i$-$t_i$-path in~$G$ induces an $s_i$-$t_i$-path in at least one DAG and each $s_i$-$t_i$-path in~$D$ is an shortest path in~$G$.

Finally, we use a $O(n+m)$-time algorithm of \citet{Tho12} for \textsc{2-Disjoint Paths} on a DAG to find $P_1$ and $P_2$.
Since there are~$O(n)$ possibilities for the point~$p$ and~$4$ possibilities for directing the edges between~$A_i$ and~$B_j$, $i \neq j$, we call~$O(n)$ instances of the algorithm of \mbox{\citet{Tho12}}.
Thus, we obtain \cref{prop:2dsp} which is formally proven in 
\cref{sec:2D-geometry}.

\subparagraph*{Generalizing to \kDSP.} 
We now discuss how to generalize the ideas from above to \DSP{k}, where $k>2$.
One central idea for $k=2$ is that the subpaths within the areas~$A_1, A_2, B_1, B_2$ (see \cref{fig:to-dag}) can hardly overlap.
The only overlap is possible along the borders.
In our approach for~$k > 2$, we simplify this even further by guessing the vertices on each path before and after the intersection (thus incurring a higher running time).
This results in four cases; see \cref{fig:areas} for an overview of the cases and the guessed vertices (marked by black squares; the guessed vertices will be called marbles latter).
\begin{figure}[t]
	\centering
	\begin{tikzpicture}[scale=.28]
		\begin{scope}
			\foreach \x in {0,1,...,10}{
				\foreach \y in {0,1,...,10}{
					\node[cross, gray!50] at (\x,\y){};
				}
			}
			\node at (0,10) {\textbf{(i)}};
			\node[2Dpoint,label=above:\textbf{\boldmath{$s_1$}}] at (0,5) {};
			\node[2Dpoint,label=above:\textbf{\boldmath{$t_1$}}] at (10,4) {};
			\node[2Dpoint,label=right:\textbf{\boldmath{$t_2$}}] at (5,1) {};
			\node[2Dpoint,label=left:\textbf{\boldmath{$s_2$}}] at (5,10) {};
			
			\diamondleftlabel{0}{5}{4}{5}{fill=red, opacity=.2}{};
			\diamondleftlabel{5}{6}{10}{4}{fill=red, opacity=.2}{};
			\diamonduplabel{5}{10}{4}{6}{fill=blue, opacity=.2}{};;
			\diamonduplabel{5}{5}{5}{1}{fill=blue, opacity=.2}{};
			
			\node[bold2Dpoint] at (4,5) {};
			\node[bold2Dpoint] at (5,6) {};
			\node[bold2Dpoint] at (4,6) {};
			\node[bold2Dpoint] at (5,5) {};
			\diamondleft{0}{5}{10}{4};
			\diamondleft{5}{1}{5}{10};
			
			\draw[thick] (4,5) -- (5,6);
			\draw[thick] (4,6) --(5,5); 
			\draw[thick] (0,5) -- (4,5);
			\draw[thick] (5,1) --(5,5); 
			\draw[thick] (5,10) --(5,8) -- (4,7) -- (4,6); 
			\draw[thick] (10,4) --(8,4) -- (6,6) -- (5,6); 
		\end{scope}
		\begin{scope}[xshift=12cm]
			\foreach \x in {0,1,...,10}{
				\foreach \y in {0,1,...,10}{
					\node[cross, gray!50] at (\x,\y){};
				}
			}
			\node at (0,10) {\textbf{(ii)}};
			\node[2Dpoint,label=below:\textbf{\boldmath{$s_1$}}] at (0,5) {};
			\node[2Dpoint,label=below:\textbf{\boldmath{$t_1$}}] at (10,4) {};
			\node[2Dpoint,label=right:\textbf{\boldmath{$t_2$}}] at (5,1) {};
			\node[2Dpoint,label=left:\textbf{\boldmath{$s_2$}}] at (5,10) {};
			
			\node[bold2Dpoint] at (2,6) {};
			\diamondleftlabel{0}{5}{2}{6}{fill=red, opacity=.2}{};
			\node[bold2Dpoint] at (6,4) {};
			\diamondleftlabel{6}{4}{10}{4}{fill=red, opacity=.2}{};
			\node[bold2Dpoint] at (6,3) {};
			\diamonduplabel{5}{1}{6}{3}{fill=blue, opacity=.2}{};;
			\node[bold2Dpoint] at (3,7) {};
			\diamonduplabel{3}{7}{5}{10}{fill=blue, opacity=.2}{};
			
			\node[bold2Dpoint] at (3,6) {};
			\node[bold2Dpoint] at (5,4) {};
			\diamondleft{0}{5}{10}{4};
			\diamondleft{5}{1}{5}{10};
			
			\draw[thick] (0,5) -- (1,5) -- (2,6) -- (3,6);
			\draw[thick, dashed] (3,6) -- (5,4) ;
			\draw[thick] (5,4) --(6,4) -- (7,5) -- (8,4) -- (10,4); 
			\draw[thick] (5,10) -- (5, 9) -- (3,7) -- (3,6) ;
			\draw[thick] (5,4) --(6,3) -- (5,2) -- (5, 1); 
		\end{scope}
		\begin{scope}[xshift=24cm]
			\foreach \x in {0,1,...,10}{
				\foreach \y in {0,1,...,10}{
					\node[cross, gray!50] at (\x,\y){};
				}
			}			
			\node at (0,10) {\textbf{(iii)}};
			\node[2Dpoint,label=below:\textbf{\boldmath{$s_1$}}] at (0,5) {};
			\node[2Dpoint,label=right:\textbf{\boldmath{$t_1$}}] at (4,6) {};
			\diamondleft{0}{5}{4}{6};
			\diamonduplabel{0}{5}{4}{6}{fill=red, opacity=.2}{};
			\node[2Dpoint,label=left:\textbf{\boldmath{$t_2$}}] at (5,2) {};
			\node[2Dpoint,label=left:\textbf{\boldmath{$s_2$}}] at (5,10) {};
			\diamondup{5}{10}{5}{2};
			\diamonduplabel{5}{10}{7}{6}{fill=blue, opacity=.2}{};
			\diamonduplabel{7}{6}{5}{2}{fill=blue, opacity=.2}{};
			\node[bold2Dpoint,label={[shift={(0.3,-.3)}]\small{\textbf{\boldmath{$\delta$}}}}] at (7,6) {};
			\draw[thick] (5,10) -- (5,9) -- (7,7) -- (7,6) -- (6,5) -- (6,4) -- (5,3) -- (5,2); 
			\draw[thick] (0,5) -- (1,5) -- (3,7) -- (4,6); 
		\end{scope}
		\begin{scope}[xshift=36cm]
			\foreach \x in {0,1,...,10}{
				\foreach \y in {0,1,...,10}{
					\node[cross, gray!50] at (\x,\y){};
				}
			}
			\node at (0,10) {\textbf{(iv)}};
			\node[2Dpoint,label=below:\textbf{\boldmath{$s_1$}}] at (0,5) {};
			\node[2Dpoint,label=right:\textbf{\boldmath{$t_1$}}] at (4,3) {};
			\diamondup{0}{5}{4}{3};
			\diamonduplabel{0}{5}{4}{3}{fill=red, opacity=.2}{};
			\node[2Dpoint,label=right:\textbf{\boldmath{$t_2$}}] at (6,4) {};
			\node[2Dpoint,label=left:\textbf{\boldmath{$s_2$}}] at (5,10) {};
			\diamondup{5}{10}{6}{4};
			\diamonduplabel{5}{10}{6}{4}{fill=blue, opacity=.2}{};
			\draw[thick] (5,10) -- (4,9) -- (7,6) -- (6,5) -- (6,4); 
			\draw[thick] (0,5) -- (1,5) -- (3,3) -- (4,3); 
		\end{scope}
	\end{tikzpicture}
	\caption{The four cases for the projection of two paths~$P_1$ and~$P_2$ in the two-dimensional grid. From left to right: 
		(i) The projection of the paths cross in one point with non-integer coordinates.
		(ii) The projection of the paths cross in at least one point with integer coordinates.
		(iii) The rectangles defined by the endpoints of~$P_1$ and~$P_2$ intersect, but the projections of~$P_1$ and~$P_2$ do not intersect. 
		(iv) The rectangles defined by the endpoints of~$P_1$ and~$P_2$ do not intersect.
	}
	\label{fig:areas}
\end{figure}
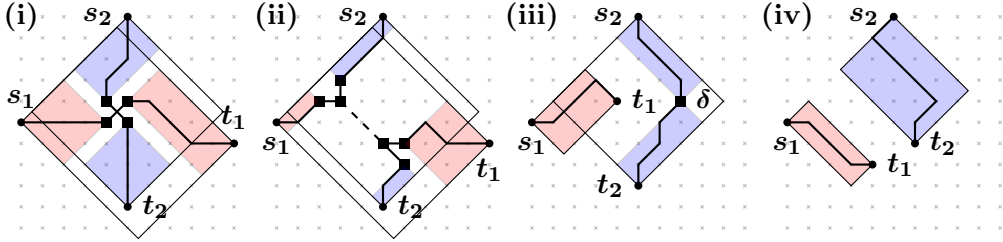
It is easy to see in \cref{fig:areas} that in each case no subpath within one red or blue area can possibly intersect with a subpath within another red or blue area.
As can be seen in \cref{fig:areas}, there is only one case where subpaths of~$P_1$ and~$P_2$ have to be computed carefully due to possible intersections: Both paths use the dashed line in case (ii).
However, this is the part where both paths are strictly monotone in both coordinates.
This is what allowed us for~$k=2$ to transform the graph into a DAG
while preserving the solutions (both paths being strictly monotone in at least one coordinate is actually sufficient for this transformation).

Considering~$k$ paths, we associate with each vertex~$v \in V$ a position in the $k$-dimensional Euclidean vector space.
For brevity, we say that a path has color~$i$ if it is strictly monotone in its~$i$-coordinate.
Thus, each path~$P_j$ has color~$j$.
The problem \textsc{$k$-Disjoint Paths on a DAG} is solvable in polynomial time for constant~$k$~\cite{FHW80}.
Thus, if we want to find subpaths from~$u_i$ to~$v_i$ for all~$i \in [k]$ that all have the same color (i.\,e.\ for each~$i \in [k]$ we have that the difference of the $i$-th coordinate of~$u_i$ and~$v_i$ is~$\dist(u_i,v_i)$), then we can use the algorithm of \citet{FHW80}.
For completeness, we provide a dynamic program with a precise running time analysis 
as \citet{FHW80} only state ``polynomial time''.
The general approach to solve the given~\mbox{\kDSP} instance is thus as follows: 
Split the paths~$P_1,\ldots,P_k$ into $f(k)$ subpaths (i.\,e. guess the endpoints of the subpaths) and find a partition of the subpaths such that
\begin{enumerate}[(i)]
		\item subpaths in the same part of the partition share a common color and hence can be computed by the algorithm of \citet{FHW80} or our dynamic program, and 
		\item subpaths in different parts of the partition cannot intersect.
\end{enumerate}
We remark that this is essentially the same general approach used by \citet{Loc21}.
However, he does not use the geometric view of the paths (as we do).
As a result, even for~$k=2$ he only bounds the number of created subpaths by~$9^{91}$ (cf.\ \citet[Lemma 4.2]{Loc21}).
While this constant is certainly not optimized, one can easily see in \cref{fig:areas} that our approach splits the two paths in at most five parts (in case (ii)).
Moreover, our geometric view allows us to use a more efficient way of splitting the paths for general~$k$, which we describe below.

\newcommand{\colvec}[3]{
	\begin{pmatrix}
		#1\\#2\\#3
	\end{pmatrix}
}

Recall that for~$k=2$ the two paths~$P_1$ and~$P_2$ have at most one intersection (a point or a straight line); see \cref{fig:areas}.
However, in three dimensions~$k > 2$ this is no longer true as neither~$P_1$ nor~$P_2$ needs to be monotone in a third dimension, see \cref{fig:3D-mess} for an example.
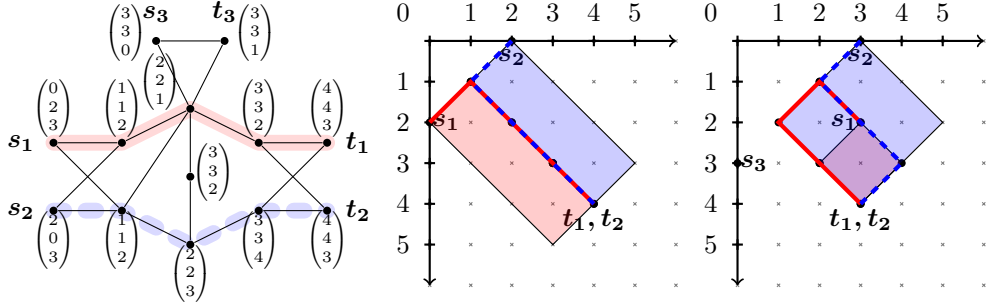
\begin{figure}[t]
	\centering
	\small
	\begin{tikzpicture}[scale=0.9]
		\begin{scope}

			\foreach[count=\i] \x / \y / \shift / \a / \b / \c in {
				0/0/{(0,-.95)}/2/0/3, 0/1/{(0,-.1)}/0/2/3, 1/0/{(0,-.95)}/1/1/2, 1/1/{(0,-.1)}/1/1/2, 2/-.5/{(0,-.95)}/2/2/3, 2/.5/{(.3,-.5)}/3/3/2, 2/1.5/{(-.4,-.2)}/2/2/1, 3/0/{(0,-.95)}/3/3/4, 3/1/{(0,-.1)}/3/3/2, 4/0/{(0,-.95)}/4/4/3, 4/1/{(0,-.1)}/4/4/3, 1.5/2.5/{(-.4,-.45)}/3/3/0, 2.5/2.5/{(.4,-.45)}/3/3/1
			}{
				\node[circle,inner sep=0pt,minimum size=3pt,fill=black,label={[shift={\shift}]{\tiny $\colvec{\a}{\b}{\c}$}}] (v\i) at (\x,\y) {};
			}
			\node[label=above:{\textbf{\boldmath{$s_3$}}}] () at (v12) {};
			\node[label=left:{\textbf{\boldmath{$s_2$}}}] () at (v1) {};
			\node[label=left:{\textbf{\boldmath{$s_1$}}}] () at (v2) {};
			\node[label=right:{\textbf{\boldmath{$t_2$}}}] () at (v10) {};
			\node[label=right:{\textbf{\boldmath{$t_1$}}}] () at (v11) {};
			\node[label=above:{\textbf{\boldmath{$t_3$}}}] () at (v13) {};

			\foreach \x/\y in {1/3, 1/4, 2/3, 2/4, 3/5, 4/7, 5/6, 5/8, 6/7, 7/9, 7/12, 8/11, 8/10, 9/10, 9/11, 3/7, 7/13, 12/13}{
				\draw (v\x) -- (v\y);
			} 
			\begin{pgfonlayer}{background}
				\draw[path,draw=blue,dash pattern=on 4pt off 8pt] (v1.center) -- (v3.center) -- (v5.center) -- (v8.center) -- (v10.center);
				\draw[path,draw=red] (v2.center) -- (v4.center) -- (v7.center) -- (v9.center) -- (v11.center);
			\end{pgfonlayer}
		\end{scope}     
		\small
		\begin{scope}[xshift=5.5cm,yshift=2.5cm,scale=.6]
			\foreach \x in {0,1,...,6}{
				\foreach \y in {0,-1,...,-6}{
					\node[cross, gray] at (\x,\y){};
				}
			}
			\drawCoordOrigin{0}{0}
		
			\foreach \x/\y in {2/0, 0/2, 1/1, 1/1, 2/2, 3/3, 2/2, 3/3, 3/3, 4/4, 4/4, 3/3}{
				\node[2Dpoint] at (\x,-\y){};
			}
			\node at (2,-.4) {\textbf{\boldmath{$s_2$}}};
			\node at (.4,-2) {\textbf{\boldmath{$s_1$}}};

			\node at (4,-4.4) {\textbf{\boldmath{$t_1,t_2$}}};
			
			\diamonduplabel{2}{0}{4}{-4}{fill=blue, opacity=.2}{};
			\diamondup{2}{0}{4}{-4}
			\diamonduplabel{0}{-2}{4}{-4}{fill=red, opacity=.2}{};
			\diamondup{0}{-2}{4}{-4}
			\draw[red, ultra thick] (0,-2) -- (1,-1) --(2,-2)--(3,-3)--(4,-4);
			\draw[blue,ultra thick, dashed] (2,0) -- (1,-1) --(2,-2)--(3,-3)--(4,-4);
		\end{scope} 
		\begin{scope}[xshift=10cm,yshift=2.5cm,scale=.6]
			\foreach \x in {0,1,...,6}{
				\foreach \y in {0,-1,...,-6}{
					\node[cross, gray] at (\x,\y){};
				}
			}
			\drawCoordOrigin{0}{0}
		
			\foreach \x/\y in {3/0, 3/2, 2/1, 2/1, 3/2, 2/3, 1/2, 4/3, 2/3, 3/4, 3/4, 0/3}{
				\node[2Dpoint] at (\x,-\y){};
			}
			\node at (3,-.4) {\textbf{\boldmath{$s_2$}}};
			\node at (.4,-3) {\textbf{\boldmath{$s_3$}}};
			\node at (2.6,-2) {\textbf{\boldmath{$s_1$}}};
			\node at (3,-4.4) {\textbf{\boldmath{$t_1,t_2$}}};

			\diamonduplabel{3}{-2}{3}{-4}{fill=red, opacity=.2}{};
			\diamondup{3}{-2}{3}{-4}

			\diamonduplabel{3}{0}{3}{-4}{fill=blue, opacity=.2}{};
			\diamondup{3}{0}{3}{-4}
			\draw[red, ultra thick] (3,-2) --(2,-1) --(1,-2)--(2,-3)--(3,-4);
			\draw[blue,ultra thick, dashed] (3,0) -- (2,-1) --(3,-2)--(4,-3)--(3,-4);
		\end{scope} 
	\end{tikzpicture} 
	\caption{
		\emph{Left side:} A graph with the coordinates displaying the distances from~$s_1$, $s_2$, and~$s_3$ 
		and two paths~$P_1$ and~$P_2$. 
		\emph{Middle:} The~$\{1,2\}$-projection of the two paths falsely suggesting only one consecutive part of overlap of the paths in~$\R^3$.
		\emph{Right side:} The~$\{2,3\}$-projection of the two paths suggesting (correctly) that there are actually two parts of overlap between the paths in~$\R^3$.
	}
	\label{fig:3D-mess}
	\vspace{-0.4cm}
\end{figure}

Thus, to exploit the properties shown in \cref{fig:areas} for two paths~$P_i$ and~$P_j$, we first project into two dimensions using the~$i$ and~$j$ coordinate.
Whenever two paths~$P_i$ and~$P_j$ intersect, then we know that the two subpaths in the intersection 
have both colors~$i$ and~$j$.
Thus, we can use for these subpaths the two-dimensional observations behind \cref{fig:areas} with new projections. %
We store for each subpath~$P'$ of~$P_i$ the set~$\Phi$ of all indices of the paths that $P'$ intersects, that is, $\Phi$ is a subset of all colors that~$P'$ has.
Now, if~$P'$ and~$P_j$ intersect, then there is a subpath of~$P_j$ that has colors~$\Phi \cup \{i\}$.
Hence this set~$\Phi$ of colors can be seen as a ``tower of colors'' that is transferred to other paths.   
Our algorithm transfers these towers from one path to another as long as possible.
These towers will be defined over permutations of subsets of~$[k]$ that encode how these color-towers are produced; see \cref{sec:towers}.
As there are at most~$k \cdot k!$ such permutations, this explains the exponent of our algorithm.
In the end, we proof our main result (\cref{thm:main-thm}) in \cref{sec:the-algorithm}.

\section{The Geometry of Two Shortest Paths}
\label{sec:2D-geometry}
We now formalize and generalize the idea behind the geometric view (visualized in \cref{fig:diamond-introduction,fig:areas}).
We start with some notation for projections.
For any~$\emptyset \subset I \subseteq [k]$ and any vector~${x \in \RR^k}$, 
we denote by~${x^I \in \RR^{|I|}}$ the orthogonal projection of~$x$ to the coordinates in~$I$.
That is, $x^I$ is the~$|I|$-dimensional vector obtained by deleting all dimensions in~$x$ that are not in~$I$.
We usually drop the brackets in the exponent, 
thus writing  
\eg $(5,6,7,8,9)^{1,3,4} \defeq (5, 7, 8)$ or~${(5, 6, 7)^2 \defeq (6)}$.
Similarly, for~$R \subseteq \RR^k$ we define $R^I \defeq \set{x^I ; x \in R} \subseteq \RR^{|I|}$.

\subsection{Coordinates and Vertices}

We associate with each vertex~$v \in V$ a vector in the $k$-dimensional Euclidean vector space.
Formally,~$\pos{v} \defeq (\pos{v}^i)_{i\in[k]} \defeq (\dis(s_i, v))_{i\in[k]} \in \NN^k$ and for $U \subseteq V$ we use~$\pos{U} \defeq \set{\pos{u} ; u \in U}$ to denote the set of all vectors of vertices in~$U$.
For a given instance of \kDSP, one can compute the vector of each vertex in~$O(km)$ time by performing a breadth-first-search from each vertex~$s_i$.

Subsequently, we will use the following notations for any index set $\emptyset \subset I \subseteq [k]$:
\[v \simeq^I w \Leftrightarrow \forall a \in I\colon \pos{v}^a \simeq \pos{w}^a,\]
for any ${\simeq} \in \{<, \le, =, \ge, >\}$ and any vertices $v, w$ and
\[V \simeq^I W \Leftrightarrow \{\pos{v}^I \mid v \in V\} \simeq \{\pos{w}^I \mid w \in W\},\]
for any ${\simeq} \in \{\subset, \subseteq, =, \supseteq, \supset\}$ and any sets~$V$ and $W$ of vertices.

\begin{lemma}
	\label{thm:basic_inequality}
	For any pair of vertices $v, w \in V$, we have $\infnorm{\pos{v} - \pos{w}} \leq \dis(v, w)$.
\end{lemma}
\begin{proof}
	Let $P=vu_1u_2\ldots u_{|P|-2}w$ be a shortest $v$-$w$-path.
	We say~$v = u_0$ and~$w=u_{|P|-1}$.
	Each edge $\{u_{j-1},u_{j}\}$ in $P$ fulfills~$\infnorm{\pos{u_{j-1}} - \pos{u_{j}}} \leq 1$ as~$|\dis(s_i,u_{j-1}) - \dis(s_i,u_j)| \leq 1$ for each vertex~$s_i$.
	Thus, by the triangle inequality, \[\infnorm{\pos{v}-\pos{w}} \leq \sum_{j=1}^{|P|-1} \infnorm{\pos{u_{j-1}}-\pos{u_{j}}} \leq \sum_{j=1}^{|P|-1} 1 = |P|-1 = \dis(v, w).\qedhere\]
\end{proof}

For two vertices~$u, w \in V$, let
\[\patharea{u}{w} \defeq \set{v \in V \mid \dis(u,v) + \dis(v,w) = \dis(u,w)}\]
be the set of all vertices that lie on a shortest~$u$-$w$-path. 
Similarly, for any~${x,y \in \NN^k}$, let
\[\patharea{x}{y} \defeq \set{z \in \RR^k \mid \infnorm{x-z} + \infnorm{z-y} = \infnorm{x-y}}\]
be the hyperrectangle spanned by~$x$ and~$y$ (whose sides form an angle of~$45^\circ$ with the coordinate axes (see~\cref{fig:diamond-introduction})).
We continue with a formal definition of \emph{colors}.

\begin{definition}\label{def:colors}
	Let $s, t$ be two vertices and let~$P$ be a shortest $s$-$t$-path.
	The pair~$(s, t)$ and the path~$P$ are~$c$-\emph{colored} if~$\dis(s, t) = \abs{\pos{s}^c-\pos{t}^c}$.
	Let
	\[{C(P) \defeq C(s, t) \defeq \set{c \in [k]; \abs{\pos{s}^c - \pos{t}^c} = \dis(s, t)}}\]
	be the set of all colors of~$P$.
	The pair~$(s, t)$ and the path~$P$ are \emph{colored} if~$C(s,t) \neq \emptyset$ and~${C(P) \neq \emptyset}$, respectively.
\end{definition}

Note that this definition of a~$c$-colored path is equivalent to saying that~$P$ is strictly monotonous in its~$c$-coordinates.
For arbitrary~$u,w \in V$ we do \emph{not} always have~$\pos{\patharea{u}{w}} \subseteq \patharea{\pos{u}}{\pos{w}}$, that is, the vectors of all vertices on a shortest~\mbox{$u$-$w$-path} are not necessarily contained in the set of vectors ``spanned'' by~$\pos{u}$ and~$\pos{w}$; see \cref{fig:3D-mess} (middle) for an example where the $s_3$-$t_3$-path does not stay within the hyperrectangle spanned by~$\pos{s_3}$ and~$\pos{t_3}$.
However, this inclusion holds for colored vertex pairs as shown next:

\begin{lemma}\label{thm:diamond_inclusion} 
	Let $v, w \in V$ be a $b$-colored pair.
	Then,~$\pos{\patharea{v}{w}} \subseteq \patharea{\pos{v}}{\pos{w}}$.
\end{lemma}
\begin{proof}
	\Wilog{}, let $v \leq^b w$.
	Let~$u$ be an arbitrary vertex in~$\patharea{v}{w}$.
	Then, $\dis(v,w) = \dis(v,u)+\dis(u,w)$. 
	\cref{def:colors,thm:basic_inequality} yield
	\begin{align*}
		\pos{w}^b 	& = \pos{v}^b + \dis(v,w) = \pos{v}^b + \dis(v,u) + \dis(u,w) 
					\ge \pos{u}^b + \dis(u,w) \ge  \pos{w}^b.
	\end{align*}
	Hence, $\pos{u}^b = \pos{v}^b + \dis(v, u)$ and~$\pos{w}^b = \pos{u}^b + \dis(u,w)$.
	\cref{thm:basic_inequality} yields~${\dis(v, u) = \infnorm{\pos{v}-\pos{u}}}$ and~$\dis(u,w) = \infnorm{\pos{v}-\pos{u}}$.
	Hence, 
	\[ \infnorm{\pos{v}-\pos{w}} = \dis(v,w) = \dis(v,u)+\dis(u,w) = \infnorm{\pos{v}-\pos{u}} + \infnorm{\pos{u}-\pos{w}}.\]
	This yields $\pos{u} \in \patharea{\pos{v}}{\pos{w}}$ and thus~$\pos{\patharea{v}{w}} \subseteq \patharea{\pos{v}}{\pos{w}}$.
\end{proof}

We will usually be concerned with the projection of $\patharea{\pos{v}}{\pos{w}}$ to some set of coordinates~$I \subseteq [k]$.
Note in this context that $\left(\patharea{\pos{v}}{\pos{w}}\right)^I = \patharea{\pos{v}^I}{\pos{w}^I}$.
Of particular interest are projections to 2-dimensional subspaces.
Recall that the area defined by $\patharea{{x}}{{y}}$ for $x,y \in \NN^2$ is a rectangle in the plane whose sides form an angle of $45^\circ$ to the coordinate axes.
The following lemma lists necessary and sufficient conditions for those rectangles to intersect.
\begin{lemma}\label{lem:patharea-coord-intersect-iff}
	Let~$x, y, \hat{x}, \hat{y} \in \RR^2$.
	Then~$\patharea{x}{y} \cap \patharea{\hat{x}}{\hat{y}} \ne \emptyset$ if and only if all of the following hold:
	\begin{enumerate}[(i)]
		\item $\min \{x^1 - x^2, y^1 - y^2\} \le \max \{\hat{x}^1 - \hat{x}^2, \hat{y}^1 - \hat{y}^2\}$, \label{rectangle-intersect-i}
		\item $\min \{\hat{x}^1 - \hat{x}^2, \hat{y}^1 - \hat{y}^2\} \le \max \{x^1 - x^2, y^1 - y^2\}$, \label{rectangle-intersect-ii}
		\item $\min \{x^1 + x^2, y^1 + y^2\} \le \max \{\hat{x}^1 + \hat{x}^2, \hat{y}^1 + \hat{y}^2\}$, and \label{rectangle-intersect-iii}
		\item $\min \{\hat{x}^1 + \hat{x}^2, \hat{y}^1 + \hat{y}^2\} \le \max \{x^1 + x^2, y^1 + y^2\}$. \label{rectangle-intersect-iv}
	\end{enumerate}
\end{lemma}

\begin{proof}
		Let~$R_1, R_2 \subseteq \NN^2$ be two axis-parallel rectangles defined by the opposite corners~${q, r \in \NN^2}$ and~$\hat{q}, \hat{r} \in \NN^2$.
	It is easy to see that~$R_1$ and~$R_2$ intersect if and only if
	\begin{enumerate}[(a)]
		\item $\min \{q^1, r^1\} \le \max \{\hat{q}^1, \hat{r}^1\}$ and $\min \{\hat{q}^1, \hat{r}^1\} \le \max \{q^1, r^1\}$ \\ 
			(\ie{} there is an overlap in the first coordinate), and
		\item $\min \{q^2, r^2\} \le \max \{\hat{q}^2, \hat{r}^2\}$ and $\min \{\hat{q}^2, \hat{r}^2\} \le \max \{q^2, r^2\}$ \\ 
			(\ie{} there is an overlap in the second coordinate).
	\end{enumerate}
	Since the intersection of two rectangles is invariant under rotation and scaling, we simply rotate~$\patharea{x}{y}$ and~$\patharea{\hat{x}}{\hat{y}}$ by 45$^\circ$ (and scale it by factor~$\sqrt{2}$) by multiplying all vectors with the matrix
	\[{\displaystyle R={\begin{bmatrix}1 &-1 \\1 &1 \\\end{bmatrix}}}.\]
	Now the above characterization for axis-parallel rectangles translates into the conditions stated in the lemma.
\end{proof}

The next lemma states that the distance between~$x$ and any~$z \in \patharea{x}{y}$ where~$x$ and~$y$ have distance~$|y^c - x^c|$ is at most~$|z^c - x^c|$.
See \cref{fig:diamond_condition} for an illustration.

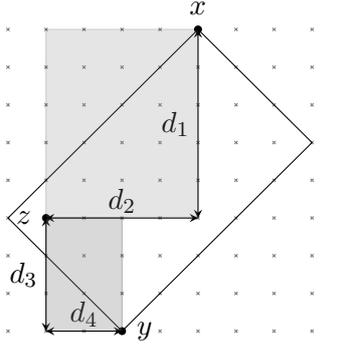
\begin{figure}
	\centering
	\begin{tikzpicture}
			\foreach \x in {-1.5,-1,...,3}{
			\foreach \y in {3,3.5,...,7}{
				\node[cross, gray] at (\x,\y){};
			}
		}

		\diamondup{0}{3}{1}{7};
		\node[2Dpoint,label=above:$x$] at (1,7) {};
		\node[2Dpoint,label=right:$y$] at (0,3) {};
		\node[2Dpoint,label=left:$z$] at (-1,4.5) {};
		\draw[<->,>=stealth] (1,7) to (1,4.5);
		\draw[<->,>=stealth] (-1,4.5) to (1,4.5);
		\draw[<->,>=stealth] (-1,3) to (0,3);
		\draw[<->,>=stealth] (-1,4.5) to (-1,3);
		\node[small,label=left:$d_1$] at (1.1,5.75) {};
		\node[small,label=below:$d_2$] at (0,5.1) {};
		\node[small,label=left:$d_3$] at (-.9,3.75) {};
		\node[small,label=below:$d_4$] at (-.5,3.6) {};
		\draw[fill=gray,opacity=0.2] (1,7) -- (1,4.5) -- (-1,4.5) -- (-1,7) -- cycle;
		\draw[fill=gray,opacity=0.3] (-1,4.5) -- (-1,3) -- (0,3) -- (0,4.5) -- cycle;
	\end{tikzpicture}
	\caption{The rectangle~$\patharea{x}{y}$ spanned by two points~$x$ and~$y$ in two dimensions and a point~${z \in \patharea{x}{y}}$.
		Note that~$\infnorm{y-x}$ is the vertical distance between~$x$ and~$y$.
		\cref{thm:diamond_condition} states that~$d_1 \geq d_2$ and~$d_3 \geq d_4$, that is, both gray rectangles are at least as high as they are wide.
	}
	\label{fig:diamond_condition}
\end{figure}
\begin{lemma}\label{thm:diamond_condition}
	Let~$b, c \in [k]$ and let~$x, y \in \NN^k$ with~$\infnorm{y-x} = y^c - x^c$.
	Then, it holds for all~$z \in \patharea{x}{y}$ that~$z^c - x^c \geq \abs{z^b - x^b} \geq 0$
	and~$y^c - z^c \geq \abs{y^b - z^b} \geq 0$.
\end{lemma}
\begin{proof} 
	By assumption and the definition of~$\patharea{x}{y}$, it holds that
	\begin{align*}
		y^c - x^c &= \infnorm{y-x} = \infnorm{y-z} + \infnorm{z-x} \\
		&\geq \abs{y^c-z^c} + \abs{z^c-x^c} \geq (y^c - z^c) + (z^c -  x^c)\\
		&= y^c - x^c.
	\end{align*}
	Thus, we have equalities everywhere.
	In particular, we have~$y^c - z^c = \infnorm{y-z} \ge |y^b - z^b|$ and~${z^c - x^c = \infnorm{z-x} \ge |z^b - x^b|}$ (as shown by the equality between the last term in the first row and the first term in the second row).
\end{proof}

We next formalize the lines we used in \cref{fig:diamond-introduction} to connect the vectors of vertices in a path.
To this end, for any path $P = v_1v_2\ldots v_i$ we define~${\zeta(P) \subset \RR^k}$ as the piecewise linear curve connecting the points of~$\pos{P}$ in the order given by $P$.
Recall that~$C(P)$ denotes the set of all colors~$b$ such that~$P$ is~$b$-colored.
The next observation states that~$\zeta(P)^{C(P)}$ is a straight line, which is equivalent to the statement that~$P$ is strictly monotone in each coordinate in~$C(P)$.

\begin{observation}\label{thm:diagonal}
	Let $P$ be a colored path.
	Then $\zeta(P)^{C(P)}$ is a straight line segment.
\end{observation}
\begin{proof}
	Let $\ell \defeq \infnorm{\pos{s_P}- \pos{t_P}}$ and $k' \defeq \abs{C(P)}$.
	The path~$P$ contains exactly~$\ell$ edges,
	each of which has an Euclidean length of at most~$\sqrt{k'}$ in the projection~$\zeta(P)^{C(P)}$.
	Hence, the length of~$\zeta(P)^{C(P)}$ is at most~$\ell \cdot \sqrt{k'}$
	which is exactly the Euclidean distance between~$\pos{s_P}^{C(P)}$ and~$\pos{t_P}^{C(P)}$.
	Thus, $\zeta(P)^{C(P)}$ is a straight line segment.
\end{proof}

\Cref{thm:diagonal} has another interesting consequence.
The intersection of two paths $P$ and~$Q$ in the~\mbox{${(C(P) \cup C(Q))}$-projection}
is always a straight line segment with an angle of~$45^\circ$ to the coordinate axes as shown in \cref{fig:diamond-introduction} (right side) and \cref{fig:areas} (case (ii)).
\begin{lemma}\label{thm:crossing_unique}
	Let~$P$ and~$Q$ be two colored paths,
	and~$C \subseteq C(P) \cup C(Q)$.
	Then~$\zeta(P)^C \cap \zeta(Q)^C$ is a (possibly empty) straight line segment.
\end{lemma}
\begin{proof}
	Note that $\zeta(P)$ and $\zeta(Q)$ are piecewise linear curves.
	According to~\cref{thm:diamond_inclusion}, for any two points $x, y \in \zeta(P)$,
	it holds that $\infnorm{x-y} = \abs{x^c - y^c}$ for all~$c \in C(P)$
	and for any two points~$x', y' \in \zeta(Q)$, it holds that~${\infnorm{x'-y'} = \abs{x'^b - y'^b}}$ for all~$b \in C(Q)$.
	Hence, for any~$z_1, z_2 \in \RR^k$ with~${z_1^C,z_2^C \in \zeta(P)^C \cap \zeta(Q)^C}$, each~$c \in C \cap C(P)$, and each~$b \in C \cap C(Q)$, it follows that 
	\[\abs{z_1^c - z_2^c} = \infnorm{z_1-z_2} = \abs{z_1^b - z_2^b}.\]
	The claim now follows analogously to the proof of~\cref{thm:diagonal}.
\end{proof}

Note that even if~$\zeta(P)^{C(P) \cup C(Q)} \cap \zeta(Q)^{C(P) \cup C(Q)}$ is non-empty, then it does not need to contain points from~$\NN^{\abs{C(P) \cup C(Q)}}$ as can be seen in \cref{fig:areas} in case (i).

\subsection{Marbles}\label{ssec:marbles-2D}

The following definition starts to formalize the notion of marbles, that is, the special vertices we guess in the different cases in \cref{fig:areas}.

\paragraph*{Crossing Paths.}
We start with the two cases in which the lines of~$P$ and~$Q$ cross (cases (i) and (ii)).

\begin{definition}\label{def:alpha-beta}
	Let~$P$ and $Q$ be two colored paths, let~$b \in C(P)$, and let~$c \in C(Q)$.
	The paths~$P$ and~$Q$ are \emph{$\{b,c\}$-crossing} if the intersection
	\[X \defeq \zeta(P)^{b,c} \cap \zeta(Q)^{b,c}\]
	is non-empty and they are~\emph{$\{b,c\}$-non-crossing} otherwise.

	If $\pos{P}^{b,c} \cap X \neq \emptyset$, 
	then we define $\alpha^{b,c}_{P}(Q)$ and~$\omega^{b,c}_{P}(Q)$ to be the first and last vertex~$v$ of~$P$ with $v^{b,c} \in X$, respectively.
	In all other cases set~$\alpha^{a,b}_P(Q) := \omega^{a,b}_P(Q) := \bot$.
	If $P$ and $Q$ are~$\{b,c\}$-crossing,
	we further define~$\partial^{b,c}_P(Q)$ and~$\varpi^{b,c}_P(Q)$ to be the last vertex before and the first vertex after that intersection, respectively.
	If no such vertex exists, we set~$\partial^{b,c}_P(Q) := \bot$ (resp.\ $\varpi^{b,c}_P(Q) := \bot$).
	In all these notations we will omit $a$, $b$, and $Q$ if they are clear from the context.
\end{definition}
Note that~$\partial^{b,c}_P(Q) := \bot$ can happen in two cases: either~$P$ and~$Q$ are $\{b,c\}$-non-crossing or~$\alpha^{b,c}_{P}(Q)$ is an endpoint of~$P$.
Moreover, subpaths between the respective~$\alpha$- and~$\omega$-vertices are by \cref{thm:crossing_unique} straight lines and~$b,c$-colored.

\begin{observation}\label{thm:crossing_colored}
If $P$ and $Q$ are two paths with $\alpha^{b,c}_P(Q) \neq \bot$, then
\[P[\alpha_P^{b,c}(Q), \omega_P^{b,c}(Q)] =^{b,c} Q[\alpha_Q^{b,c}(P), \omega_Q^{b,c}(P)] =^{b,c} (\patharea{\pos{\alpha_P^{b,c}(Q)}}{\pos{\omega_P^{b,c}(Q)}}) \cap \NN^k.\]
In particular, both of these subpaths are $b,c$-colored.
\end{observation}

It remains to consider the subpaths between~$s$- and~$\alpha$-vertices and between~$\omega$- and~$t$-vertices.
By~\cref{thm:diamond_inclusion}, these have to lie in the hyperrectangle areas
\begin{equation}
\patharea{\pos{s_P}}{\pos{\partial^{b,c}_P(Q)}}, \patharea{\pos{\varpi^{b,c}_P(Q)}}{\pos{t_P}}, \patharea{\pos{s_Q}}{\pos{\partial^{b,c}_Q(P)}}, \text{ and }\patharea{\pos{\varpi^{b,c}_Q(P)}}{\pos{t_Q}}. \label{eq:areas}
\end{equation}
Cases (i) and (ii) in \cref{fig:areas} suggest that these four areas are pairwise disjoint.
We will show that this is indeed the case and, to this end, we show the following two observations.
The first one states that the~$\partial$-vertex on~$P$ has a~$b$-coordinate that is at most the~$b$-coordinate of the~$\partial$-vertex on~$Q$, where~$b$ is the ``original'' color of~$P$.
Note that this~$\partial$-vertex is right before the respective~$\alpha$-vertex or before the single crossing point with non-integer coordinates.
Since~$P$ is strictly~$b$-monotone, the path~$Q$ can at most increase or decrease as fast as~$P$ from the point of intersection.

\begin{observation}\label{obs:partial-a-b-relation} 
	Let~$P$ and $Q$ be two $\{b,c\}$-crossing paths with~$\partial_P^{b,c}(Q), \varpi_P^{b,c}(Q) \neq \bot$. 
	If~$P$ is a subpath of a shortest~$s_b$-$t_b$-path and~$Q$ is a subpath of a shortest~$s_c$-$t_c$-path, then
	\begin{align*}
	\pos{\partial^{b,c}_P(Q)}^b \leq \pos{\partial^{b,c}_Q(P)}^b,\ \pos{\partial^{b,c}_Q(P)}^c \leq \pos{\partial^{b,c}_P(Q)}^c,\ \pos{\varpi^{b,c}_P(Q)}^b \geq \pos{\varpi^{b,c}_Q(P)}^b,  \text{ and }\pos{\varpi^{b,c}_Q(P)}^c \geq \pos{\varpi^{b,c}_P(Q)}^c.
	\end{align*}
\end{observation}
\begin{proof} 
	Let~$z \in \zeta(P)^{b,c} \cap \zeta(Q)^{b,c} \subseteq \RR^2$ have minimal $b$-coordinate.
	Note that
	\[\infnorm*{z - \pos{\partial^{b,c}_P(Q)}^{b,c}} = \infnorm*{z - \pos{\partial^{b,c}_Q(P)}^{b,c}}.\]
	Since $\zeta(P)$ is strictly increasing in its $b$-coordinate, we can infer that
	\[z^b - \pos{\partial^{b,c}_P(Q)}^b = \infnorm*{z - \pos{\partial^{b,c}_P(Q)}^{b,c}} = \infnorm*{z - \pos{\partial^{b,c}_Q(P)}^{b,c}} \geq z^b - \pos{\partial^{b,c}_Q(P)}^b.\]
	This yields $\pos{\partial^{b,c}_P(Q)}^b \leq \pos{\partial^{b,c}_Q(P)}^b$ and the other three inequalities follow analogously.
\end{proof}

The second observation is a simple but useful reformulation of \cref{thm:basic_inequality}.

\begin{observation}\label{obs:partial-a-b-to-s_a-relation}
	Let~$b,c \in [k]$ and let~$(v, w)$ be a $b$-colored pair of vertices with $v <^b w$.
	Then~$\pos{w}^b - \pos{w}^c \ge \pos{v}^b - \pos{v}^c$.
\end{observation}
\begin{proof}
	By \cref{thm:basic_inequality}, $\pos{w}^c - \pos{v}^c \leq \dis(v, w) = \pos{w}^b - \pos{v}^b$.
	A simple arithmetic reformulation yields~$\pos{w}^c -\pos{w}^b \leq \pos{v}^c - \pos{v}^b$ and multiplying both sides with~$-1$ completes the proof.
\end{proof}

We are now in the position to prove the statement that the four areas defined in Term~(\ref{eq:areas}) are pairwise disjoint.

\begin{lemma}\label{lem:four-diamonds}
		Let $P$ and $Q$ be two $\{b,c\}$-crossing paths.
	The sets
	\[\left(\patharea{\pos{s_P}}{\pos{\partial^{b,c}_P(Q)}}\right)^{b,c}\hspace*{-.1cm}, \left(\patharea{\pos{\varpi^{b,c}_P(Q)}}{\pos{t_P}}\right)^{b,c}\hspace*{-.1cm}, \left(\patharea{\pos{s_Q}}{\pos{\partial^{b,c}_Q(P)}}\right)^{b,c}\hspace*{-.1cm}, \text{\;and } \left(\patharea{\pos{\varpi^{b,c}_Q(P)}}{\pos{t_Q}}\right)^{b,c} \]
	are pairwise disjoint (or undefined).
\end{lemma}
\begin{proof}
	\Wilog, 
	let~$P$ be $b$-colored,
	let~$Q$ be~$c$-colored,
	let~$s_P <^b t_P$, and
	let~$s_Q <^c t_Q$.
	Recall that~$s_P$ and~$t_P$ are the start and end vertices of~$P$, respectively.
	We further assume that all above sets are defined, that is, none of the described end points is~$\bot$.
	By \cref{thm:diamond_condition}, for any $x \in \patharea{\pos{s_P}}{\pos{\partial_P}}$ and $y \in \patharea{\pos{\varpi_P}}{\pos{t_P}}$
	it holds that~$x^b \leq \pos{\partial_P}^b < \pos{\varpi_P}^b \leq  y^b$,
	and thus~$\left(\patharea{\pos{s_P}}{\pos{\partial_P}}\right)^{b,c} \cap \left(\patharea{\pos{\varpi_P}}{\pos{t_P}}\right)^{b,c} = \emptyset$.
	An analogous argument holds for~$\left(\patharea{\pos{s_P}}{\pos{\partial_P}}\right)^{b,c}$ and~$\left(\patharea{\pos{\varpi_P}}{\pos{t_P}}\right)^{b,c}$.
	
	We will now use~\cref{lem:patharea-coord-intersect-iff} to show that~$\left(\patharea{\pos{s_P}}{\pos{\partial_P}}\right)^{b,c} \cap \left(\patharea{\pos{s_Q}}{\pos{\partial_Q}}\right)^{b,c} = \emptyset$.
	Since all other remaining cases are analogous, this will conclude the proof.
	By \cref{obs:partial-a-b-to-s_a-relation}, it holds that
	\begin{align*}
	\max\{\pos{s_P}^b - \pos{s_P}^c, \pos{\partial_P}^b - \pos{\partial_P}^c\} &=  \pos{\partial_P}^b - \pos{\partial_P}^c \text{ and}\\
	\min\{\pos{s_Q}^b - \pos{s_Q}^c, \pos{\partial_Q}^b - \pos{\partial_Q}^c\} &= \pos{\partial_Q}^b - \pos{\partial_Q}^c.
	\end{align*}
	Observe that~$\pos{\partial_P}^{b,c} \neq \pos{\partial_Q}^{b,c}$ as otherwise~$\partial_P$ would lie on the intersection of~$P$ and~$Q$, a contradiction to the definition of~$\partial_P$.
	Hence~$\pos{\partial_P}^{b} \neq \pos{\partial_Q}^{b}$ or~$\pos{\partial_P}^{c} \neq \pos{\partial_Q}^{c}$.
	In the former case, \cref{obs:partial-a-b-relation} states that~$\partial_Q >^b \partial_P$ and in the latter case it states~$\partial_P >^c \partial_Q$.
	Hence, $\pos{\partial_P}^c + \pos{\partial_Q}^b > \pos{\partial_P}^b + \pos{\partial_Q}^c$
	and thus $\pos{\partial_P}^b - \pos{\partial_P}^c < \pos{\partial_Q}^b - \pos{\partial_Q}^c$.
		
	Setting~$x = \pos{s_P}^{b,c}$, $y = \pos{\partial_P}^{b,c}$, $\hat{x} = \pos{s_Q}^{b,c}$, and~$\hat{y} = \pos{\partial_Q}^{b,c}$ in \cref{lem:patharea-coord-intersect-iff} then violates condition~\ref{rectangle-intersect-ii} and hence~$\left(\patharea{\pos{s_P}}{\pos{\partial_P}}\right)^{b,c} \cap \left(\patharea{\pos{s_Q}}{\pos{\partial_Q}}\right)^{b,c} = \emptyset$.
\end{proof}

\paragraph*{Noncrossing Paths.}

We continue with the definition of marbles for the cases where the two paths~$P$ and~$Q$ are~$\{b,c\}$-non-crossing.
\cref{fig:areas} shows that even in this case~$\patharea{s_P}{t_P}$ and~$\patharea{s_Q}{t_Q}$ in general are not disjoint (case (iii)).
Since the two bottom cases are distinguished by the intersection of~$\patharea{s_P}{t_P}$ and $\patharea{s_Q}{t_Q}$, we start with a definition of this intersection.

\begin{definition} \label{def:common-area}
	Let~$b,c \in [k]$.
	Let~$P$ be a~$b$-colored path and let~$Q$ be a~$c$-colored path.
	The \emph{common $b,c$-area} of~$P$ and~$Q$ is
	\[\Delta^{b,c}(P,Q) \defeq (\patharea{\pos{s_P}}{\pos{t_P}})^{b,c} \cap (\patharea{\pos{s_Q}}{\pos{t_Q}})^{b,c}.\]
\end{definition}

Note that if~$\Delta^{b,c}(P,Q) = \emptyset$, then by \cref{thm:diamond_inclusion}, they do not share vertices with common vectors and hence they do not share common vertices.
If~$\Delta^{b,c}(P,Q) \neq \emptyset$ and~$P$ and~$Q$ are~$\{b,c\}$-crossing, then we can use \cref{def:alpha-beta} to define the marbles.
It hence remains to study the case where~$\Delta^{b,c}(P,Q) \neq \emptyset$ and~$P$ and~$Q$ are~$\{b,c\}$-non-crossing.
In this case we need at most one marble per path and this marble is defined as follows.

\begin{definition}
	\label{def:delta}
	Let $P$ be a $b$-colored path, $Q$ be a $c$-colored path, and let \wilog{} be~$s_Q <^b t_Q$.
	Define
	\[B \defeq \set{v \in V \mid v=^b s_Q \land v <^c s_Q} \cup \set{v \in V \mid v=^b t_Q \land v>^c t_Q}.\]
	If~$P \cap B \neq \emptyset$, then $\delta^{b,c}_P(Q)$ is the unique vertex in~$P \cap B$.
	If~$P \cap B = \emptyset$, then~$\delta^{b,c}_P(Q) = \bot$.
\end{definition}

\begin{observation}
	$\delta^{b,c}_P(Q)$ is well-defined.
\end{observation}
\begin{proof}
	Since $P$ is strictly monotone in the $b$-coordinate, it can clearly contain at most one point from 
	$B_1 \defeq \set{v \in V; v=^b s_Q \land v <^c s_Q}$ and one from~${B_2 \defeq \set{v \in V; v=^b t_Q \land v>^c t_Q}}$.
	It remains to show that it cannot intersect both sets.
	To this end, observe that \cref{thm:diamond_condition} states for any~$b_1 \in B_1$ and~$b_2 \in B_2$ that~$\abs{b_1^c - b_2^c} > \abs{s_Q^b - t_Q^b} \geq \abs{s_Q^b - t_Q^b} = \abs{b_1^b - b_2^b}$, and
	therefore the pair~$(b_1, b_2)$ is not $b$-colored and thus~$P$ cannot contain both~$b_1$ and~$b_2$.
\end{proof}

We next show that if~$\Delta^{b,c}(P,Q) \neq \emptyset$ and~$P$ and~$Q$ are~$\{b,c\}$-non-crossing, then at least one of the vertices~$\delta^{b,c}_P(Q)$ or $\delta^{b,c}_Q(P)$ exists.
Afterwards we will show that these vertices guarantee that the respective new areas are disjoint.

\begin{lemma}\label{lem:delta_exists}
	Let~$P$ and $Q$ be two~$\{b,c\}$-non-crossing paths with~$\Delta^{b,c}(P,Q) \neq \emptyset$.
	
	Then, ${\delta^{b,c}_P(Q) \neq \bot}$ or $\delta^{b,c}_Q(P) \neq \bot$.
\end{lemma}
\begin{proof}
	Note that the definition of~$\Delta^{b,c}(P,Q)$ requires that \wilog{}~$P$ is~$b$-colored and~$Q$ is~$c$-colored.
	We further assume \wilog{} that~$\{b,c\} = [2]$ and that~$s_P <^b t_P$ and $s_Q <^c t_Q$.
	
	Suppose towards a contradiction that~$\delta^{b,c}_P(Q) = \delta^{b,c}_Q(P) = \bot$.
	Since $P$ and $Q$ are~$\{b,c\}$-non-crossing and $\delta^{b,c}_Q(P) = \bot$, the path $Q$ cannot cross the curve
	\[\set{x \in \RR^2; x^c = s_P \land x^b < s_P} \cup \zeta(P)^{b,c} \cup \set{x \in \RR^2; x^c = t_P \land x^b > t_P}.\]
	Since~$Q$ cannot cross the line, it is located completely on one side of it.
	Assume \wilog{} that $Q$ (and thus in particular $t_Q$) is located on the side containing $(0, 0)$, that is, for all~$v$ in~$P$ and~$w$ in~$Q$ with~$v^b = w^b$ it holds that~$v^c > w^c$.

	Then there are three possible cases: $t_Q^b < s_P^b$,~$t_Q^b \in [s_P^b,t_P^b]$, or~$t_Q^b > t_P^b$.
	Note that in the first case by \cref{thm:diamond_condition} it holds for any~$z \in \Delta^{b,c}(P,Q)$ that
	\[z^c - t_Q^c \geq z^b - t_Q^b > z^b - s_P^b \geq z^c - s_P^c,\]
	a contradiction to~$z \in (\patharea{\pos{s_Q}}{\pos{t_Q}})^{b,c}$.
	In the last case, a similar argument holds with
	\[s_P^b < t_Q^b - z^b \leq t_Q^c - z^c < s_P^c - z^c,\]
	which is a contradiction to~$z \in (\patharea{\pos{s_P}}{\pos{t_P}})^{b,c}$.
	It remains to analyze the case where~${t_Q^b \in [s_P^b,t_P^b]}$.
	Note that in this case there is a vertex~$p$ in~$P$ with~$p^b = t_Q^b$.
	Note that~$t_Q^c < p^c$ and therefore~${p \in \set{v \in V; v=^b t_Q \land v>^c t_Q}}$.
	Thus,~$\delta^{b,c}_P(Q) = p \neq \bot$, a contradiction.
\end{proof}

The next lemma shows that if $P$ and $Q$ are~$\{b,c\}$-non-crossing, but they have a common area $\Delta_{b,c}$ and~$\delta_P^{b,c}(Q) \neq \bot$, then the area between~$s_Q$ and~$t_Q$ is disjoint from the two areas between~$s_P$ and~$\delta_P^{b,c}(Q)$ and between~$\delta_P^{b,c}(Q)$ and~$t_P$.
Hence~$\delta_P^{b,c}(Q)$ is the last type of marble needed.

\begin{lemma}\label{lem:delta_disjointness}
	Let~$P$ be a $b$-colored path
	and let~$Q$ a $c$-colored path such that~$\Delta_{b,c}(P,Q) \neq \bot$ and~$\delta_P^{b,c}(Q) \neq \bot$.
	Then,
	\[\left(\patharea{\pos{s_Q}}{\pos{t_Q}}\right)^{b,c} \text{ is disjoint from } \left(\patharea{\pos{s_P}}{\pos{\delta_P^{b,c}(Q)}}\right)^{b,c} \cup \left(\patharea{\pos{\delta_P^{b,c}(Q)}}{\pos{t_P}}\right)^{b,c}.\]
\end{lemma}

\begin{proof}
	For the sake of readability, we use~$\delta \defeq \delta_P^{b,c}(Q)$.
	We will show that
	\[\left(\patharea{\pos{s_Q}}{\pos{t_Q}}\right)^{b,c} \cap \left(\patharea{\pos{s_P}}{\pos{\delta}}\right)^{b,c} = \emptyset.\]
	The proof for~$(\patharea{\pos{\delta}}{\pos{t_P}})^{b,c}$ is then completely analogous.
	Assume \wilog{} that~$\delta =^b t_Q$ and~$\delta >^c t_Q$.
	Notice that
	\begin{align*}
		\max\{\pos{s_P}^b - \pos{s_P}^c, \pos{\delta}^b - \pos{\delta}^c\} 
			\overset{\text{Obs.~\ref{obs:partial-a-b-to-s_a-relation}}}{=}& \pos{\delta}^b - \pos{\delta}^c \\
			<  \pos{t_Q}^b - \pos{t_Q}^c
			\overset{\text{Obs.~\ref{obs:partial-a-b-to-s_a-relation}}}{=}& \min\{\pos{s_Q}^b - \pos{s_Q}^c, \pos{t_Q}^b - \pos{t_Q}^c\}.
	\end{align*}
	Setting~$x \defeq \pos{s_P}^{b,c}$, $y \defeq \pos{\delta}^{b,c}$, $\hat{x} \defeq \pos{s_Q}^{b,c}$, and~$\hat{y} \defeq \pos{t_Q}^{b,c}$ in \cref{lem:patharea-coord-intersect-iff} then yields that condition~\ref{rectangle-intersect-ii} is violated and thus~$(\patharea{\pos{s_P}}{\pos{\delta}})^{b,c} \cap (\patharea{\pos{s_Q}}{\pos{t_Q}})^{b,c} = \emptyset$.
\end{proof}

\subsection{Optimized Algorithm for Finding 2 Disjoint Shortest Paths}

With the basics on crossing and noncrossing paths, we now have all ingredients for the proof of \cref{prop:2dsp}.

\dsptwoThm*
\begin{proof}%
		Let $I :=\left(G=(V,E),k,((s_1,t_1),(s_2,t_2)\right)$ be an instance of \DSP{2}.
		Compute $\pos{v}$ for all~${v \in V}$ via two breadth-first searches in $O(n+m)$ time.
		We assume without loss of generality that $G$ is connected.
		To ensure that we report $I$ being a \yes-instance only if $I$ is indeed a \yes-instance, 
		we perform a sanity-check in the very end to verify that our guesses were correct.
		Note that such a sanity-check can be done in $O(m)$ time.
		Hence, we only need to show that we find in $O(nm)$ time a solution to $I$ if there is one.
		To this end, assume there are disjoint shortest~$s_i$-$t_i$-paths $P_i$ for $i \in [2]$. 
		By \cref{thm:crossing_unique}, we have three cases for~$\zeta(P_1) \cap \zeta(P_2)$:
		It is either empty, contains only points with non-integer coordinates, or it contains at least one point with only integer coordinates.

		(Case 1): $\zeta(P_1) \cap \zeta(P_2)$ is empty.
 		If $\Delta^{1,2}(P_1, P_2) = \emptyset$, then, by \cref{thm:diamond_inclusion}, a solution can easily be found by two independent breadth-first-searches.
		Otherwise, we guess in $O(n)$ time the vertex $\delta^{1,2}_{P_1}(P_2)$ on $P_1$ or~$\delta^{2,1}_{P_2}(P_1)$ on $P_2$ from \cref{def:delta}.
		By \cref{lem:delta_exists}, at least one of them exists (assume without loss of generality that~$\delta^{1,2}_{P_1}(P_2)$ exists). 
		By \cref{lem:delta_disjointness}, $\delta^{1,2}_{P_1}(P_2) \neq \bot$ and 
		any shortest~$s_1$-$\delta^{1,2}_{P_1}(P_2)$-path ($\delta^{1,2}_{P_1}(P_2)$-$t_1$-path) is vertex disjoint from 
		any shortest~$s_2$-$t_2$-path.
		Hence, we now can check in $O(m)$ time whether $I$ is a \yes-instance.

		(Case 2): $\zeta(P_1) \cap \zeta(P_2)$ has no point with integer coordinates.
		Then, we guess the four points surrounding $\zeta(P_1) \cap \zeta(P_2)$.
		This can be done in~$O(n)$ time by guessing the vertex~$\partial^{1,2}_{P_1}(P_2)$ and by guessing whether~$\varpi^{1,2}_{P_1}(P_2)$ has a higher or lower~$2$-coordinate than~$\partial^{1,2}_{P_1}(P_2)$.
		Let~$p_{s_i}$ and~$p_{t_i}$ be the guessed points used by~$P_i$ 
		such that~$p_{s_i}^i + 1 = p_{t_i}^i$, for~$i \in [2]$.
		Now we construct a directed graph $D$ on the vertices~$V$
		such that there is an arc~$(v,w)$ if 
		$\{v,w\} \in E$, $\pos{v}^i +1 = \pos{w}^i$,
		and~$\pos{\{v,w\}} \subseteq \patharea{\pos{s_i}}{p_{s_i}} \cup \patharea{p_{t_i}}{\pos{t_i}}$ for some $i \in [2]$.
		Note that $D$ is acyclic and that each $s_i$-$t_i$-path in $D$ corresponds (same set of vertices) to a shortest $s_i$-$t_i$-path in $G$, and
		has an arc $(v,w)$ such that $\pos{v} = p_{s_i}$ and  $\pos{w} = p_{t_i}$.
		Hence, by \cref{lem:four-diamonds} we can simply use two breadth-first-searches from~$s_1$ and~$s_2$ to find a solution.

		(Case 3): $\zeta(P_1) \cap \zeta(P_2)$ has at least one point with integer coordinates.
		Without loss of generality, 
		we assume that $\pos{\alpha_{P_1}^{1,2}(P_2)} = \pos{\alpha_{P_2}^{1,2}(P_1)}$, 
		otherwise we swap one terminal pair in the input instance.
		We guess in $O(n)$ time the discrete point~$p \in \zeta(P_1)^{1,2} \cap \zeta(P_2)^{1,2}$ 
		such that~$\pos{p}^1$ is minimized.
		Let $A_i := \patharea{\pos{s_i}}{\pos{p}}$ and 
		$B_i := \patharea{\pos{p}}{\pos{t_i}}$, for all $i \in [2]$.
		Now we construct a directed acyclic graph $D$ on the vertices $V$
		such that there is an arc $(v,w)$ if for some $i\in[2]$
		we have
		(1)~$\{v,w\} \in E$, 
		(2)~$\pos{v}^i +1 = \pos{w}^i$,
		and 
		(3)~(a) $\pos{v} \in A_i \setminus B_i $ and
		$\pos{w} \in A_i$ or (b)
		$\pos{v} \in B_i$ and~$\pos{w} \in B_i \setminus A_i$.

		To add the edges with coordinates in $A_1 \cap B_2,A_2 \cap B_1$ to $D$, 
		we observe that our assumption implies 
		that $\zeta(P_1) \cap \zeta(P_2) \subseteq B_1 \cap B_2$.
		Hence, all edges where the vertices have coordinates in~$A_i \cap B_j$ can only be used by either $P_1$ or $P_2$, for each $\{i,j\} = [2]$.
		Thus, we branch in four cases and add the edges accordingly.
		Note that~$D$ is acyclic and that each~$s_i$-$t_i$-path in~$D$ corresponds to a shortest $s_i$-$t_i$-path in $G$ going through 
		point $p$.
		Hence, by \cref{lem:four-diamonds} $I$ is a \yes-instance if and only if
		there are disjoint $s_i$-$t_i$-path in~$D$, for all $i \in [2]$.
		Thus, we apply an~$O(n+m)$-time algorithm of \citet{Tho12} for \textsc{2-Disjoint Paths} on a DAG.
		This yields a total running time of~$O(nm)$.
\end{proof}

\section{The Geometry of Many Shortest Paths} \label{sec:towers}

In the previous section, we looked at two shortest paths~$P$ and~$Q$ from~$s_P$ to~$t_P$ and~$s_Q$ and~$t_Q$, respectively.
We showed that we could split~$P$ and~$Q$ in at most three parts each such that these parts are either geometrically separated or can be handled together through a reduction to DAGs.
The geometric separation ensured that we could find subpaths that are guaranteed to not intersect with the other subpaths.
In this section, we generalize this approach to more than two paths.
To this end, we start with a formalization of the splitting of the paths and then transfer the separation argument to this notion.

\subsection{From Geometry To Shortest Paths} 

We plan to refine the search for shortest paths by guessing a few vertices on the shortest paths. 
Formally, this is captured in the following definition of \mpath{s}.

\begin{definition}
	\label{def:mpaths}
	An \emph{$i$-\mpath}~$T$ is a set of vertices such that for each~$u,v \in T$ the pair~$(u,v)$ is $i$-colored.
	The unique vertices with minimal resp.\ maximal $i$-coordinate in~$T$ are $\start(T)$ resp.\ $\en(T)$.
	A path $P$ \emph{follows}~$T$ if~$T \subseteq P$.
	An $i$-\mpath{}~$S$ is called a \emph{segment} of $T$ if~$S \subseteq T$ and $S$ contains all vertices $v \in T$ with $\start(S) <^i v <^i \en(S)$.
\end{definition}

Our algorithmic approach is as follows: 
For each vertex pair~$(s_i,t_i)_{i\in[k]}$ we maintain a \mpath{}, starting with~$\{s_i,t_i\}$.
Then, more and more vertices will be guessed (added to the \mpath{}) with the goal to find shortest paths for the segments and to combine them.
In order to generalize the geometric arguments that allowed us in \cref{sec:2D-geometry} to freely combine subpaths, we adjust our definition of intersection a bit.
More precisely, we still look at the area spanned by the endpoints of minimal segments in a \mpath{}.
However, two ``consecutive'' segments clearly overlap in their end respectively start point.
To accommodate this, we use the following notion which allows the rectangles to overlap in endpoints and in non-integer coordinates (as the later do not correspond to coordinates of vertices).

\begin{definition}[$I$-avoiding]
	\label{def:avoiding}
	Let $\emptyset \subset I \subseteq [k]$.
	Two vertex pairs~$(s_p,t_p)$ and~$(s_q,t_q)$ are \emph{$I$-avoiding} if
	\[\NN^k \cap (\patharea{\pos{s_p}^I}{\pos{t_p}^I}) \cap (\patharea{\pos{s_q}^I}{\pos{t_q}^I}) \subseteq \{\pos{s_p}^I,\pos{t_p}^I\} \cap \{\pos{s_q}^I, \pos{t_q}^I\}\]
	and they are \emph{avoiding} if they are~$[k]$-avoiding.	
\end{definition}

Note that if two vertex pairs~$(s_p,t_p)$ and~$(s_q,t_q)$ are $I$-avoiding for some~$I \subset [k]$, then they are also~$I'$-avoiding for any~$I' \supset I$.
Thus, if they are~$I$-avoiding for some~$I \subseteq [k]$, then they are avoiding.

As in \cref{sec:2D-geometry}, the above geometric notion for separating paths ensures vertex-disjointness of the shortest paths.
However, it only ensures that paths are internally vertex disjoint, that is, we still allow paths to intersect but they may only intersect in vertices that are endpoints of both paths.

\begin{observation}\label{obs:avoiding-implies-disjoint}
	Let $P$ be an $s_P$-$t_P$-path and $Q$ be an $s_Q$-$t_Q$-path with~$(s_P,t_P)$ and~$(s_Q,t_Q)$ being $I$-avoiding for some~$I \subseteq [k]$.
	Then $P$~is internally vertex-disjoint from~$Q$.
\end{observation}

We transfer the notion of avoiding in a natural way to \mpath{s}.

\begin{definition}
	A \mpath{} is called \emph{minimal} if it contains exactly two vertices and it is called~\emph{$j$-colored} if $(\start(T), \en(T))$ is $j$-colored.
	Let~$I \subseteq [k]$.
	Two minimal marble paths~$T$ and~$T'$ are \emph{$I$-avoiding} if $(\start(T), \en(T))$ and $(\start(T'), \en(T'))$ are $I$-avoiding.
	Two \mpath{s} are \emph{$I$-avoiding} if all their minimal segments are pairwise $I$-avoiding.
\end{definition}

An elementary but important observation is that ``refining'' a \mpath{}, i.e., adding intermediate vertices, will not destroy the avoiding property.
Intuitively this is clear as adding further vertices shrinks the area in which the \mpath{} can lie.

\begin{lemma}
	\label{thm:supersegments-avoiding}
	Let~$S$ and~$S'$~be $i$-\mpath{s} and let $T$ and~$T'$~be $j$-\mpath{s} with
	\begin{align*}
	S' &\supseteq S, & T' &\supseteq T,\\
	\start(S') &= \start(S), & \start(T') &= \start(T),\\
	\en(S') &= \en(S),\text{ and } &\en(T') &= \en(T).
	\end{align*}
	If $S$ and~$T$ are $I$-avoiding, then so are $S'$ and~$U'$.
\end{lemma}

\begin{proof}
	Let~$\bar{S'}$ be a minimal segment of~$S'$ and~$\bar{S}$ be the minimal segment of~$S$ that contains $\bar{S'}$, that is, $\start(\bar{S}) \le^i \start(\bar{S'})$ and~$\en(\bar{S}) \ge^i \en(\bar{S'})$.
	Let~$\bar{T'}$ and~$\bar{T}$ be analogously defined for~$T$ and~$T'$.
	We need to show that $\bar{S'}$ and~$\bar{T'}$ are avoiding.
	
	As $S'$ is $i$-colored and $T'$ is $j$-colored, it follows from the choice of~$\bar{S}$ and~$\bar{T}$ that
	\begin{align*}
		\patharea{\pos{\start(\bar{S})}^I}{\pos{\en(\bar{S})}^I} & \supseteq \patharea{\pos{\start(\bar{S'})}^I}{\pos{\en(\bar{S'})}^I} \\
		\patharea{\pos{\start(\bar{T})}^I}{\pos{\en(\bar{T})}^I} & \supseteq \patharea{\pos{\start(\bar{T'})}^I}{\pos{\en(\bar{T'})}^I}.
	\end{align*}
	Thus, $\bar{S'}$ and~$\bar{T'}$ are avoiding as~$S$ and~$T$ are.
\end{proof}

We now describe which vertices we add to a \mpath{} in the case of~$k=2$.
To keep the arguments as simple as possible, we avoid optimizations as in \cref{prop:2dsp}.
The reason is that the case~$k \ge 3$ uses the following definition and lemmas as base case in inductive arguments.

We start with defining the set of marbles for a pair of paths.

\begin{definition}\label{def:2d-crossing-set}
	Let $P$ be a $b$-colored $s_P$-$t_P$-path and $Q$ be a $c$-colored $s_Q$-$t_Q$-path.
	The set of~\emph{$\{b,c\}$-marbles} of~$P$ with respect to~$Q$ is
	\[ \cross_P^{b,c}(Q) \defeq \set{s_P, t_P, \mu^{b,c}_P(Q) ; \mu \in \{\alpha, \omega, \partial, \varpi, \delta\} } \setminus \{\bot\} . \]
\end{definition}

The next two lemmas state that the \mpath{s} defined in \cref{def:2d-crossing-set} are generally avoiding, except for the segments located between the respective~$\alpha$- and $\omega$-vertices.
This is essentially a translation of our results in \cref{ssec:marbles-2D} to \mpath{s}.

First we deal with the case where no such segments exist.

\begin{lemma}
	\label{lem:no-alpha-means-avoiding}
	Let~$P$ be a~$b$-colored~\mbox{$s_P$-$t_P$-path} and let~$Q$ be a~$c$-colored~\mbox{$s_Q$-$t_Q$-path}.
	
	If~${\alpha^{b,c}_P(Q) = \bot}$, then the \mpath{s}~$\cross^{b,c}_P(Q)$ and~$\cross^{b,c}_Q(P)$ are avoiding.
\end{lemma}

\begin{proof}
	We consider the two cases whether or not~$P$ and~$Q$ are~$\{b,c\}$-crossing.
	If~$P$ and~$Q$ are~$\{b,c\}$-crossing, then, since $\alpha^{b,c}_P(Q) = \bot$, it follows that (see \cref{fig:areas} case (i))
	\[P = P[s_P,\partial^{b,c}_P(Q)] \bullet P[\varpi^{b,c}_P(Q),t_P] \text{ and } Q = Q[s_Q,\partial^{b,c}_Q(P)] \bullet Q[\varpi^{b,c}_Q(P),t_Q].\]
	Since $\alpha^{b,c}_P(Q) = \bot$, the segments~$\{\partial^{b,c}_P(Q),\varpi^{b,c}_P(Q)\}$ and~$\{\partial^{b,c}_Q(P),\varpi^{b,c}_Q(P)\}$ are avoiding.
	Moreover, by \cref{lem:four-diamonds}, the sets
	\[\left(\patharea{\pos{s_P}}{\pos{\partial^{b,c}_P(Q)}}\right)^{b,c}\hspace*{-.1cm}, \left(\patharea{\pos{\varpi^{b,c}_P(Q)}}{\pos{t_P}}\right)^{b,c}\hspace*{-.1cm}, \left(\patharea{\pos{s_Q}}{\pos{\partial^{b,c}_Q(P)}}\right)^{b,c}\hspace*{-.1cm}, \text{\;and } \left(\patharea{\pos{\varpi^{b,c}_Q(P)}}{\pos{t_Q}}\right)^{b,c} \]
	are pairwise disjoint (or undefined).
	Thus, the three minimal segments of~$\cross^{b,c}_P(Q)$ avoid each of the three minimal segments of~$\cross^{b,c}_Q(P)$.
	
	If~$P$ and~$Q$ are $\{b,c\}$-non-crossing, then consider~$\Delta^{b,c}(P,Q) = (\patharea{\pos{s_P}}{\pos{t_P}})^{b,c} \cap (\patharea{\pos{s_Q}}{\pos{t_Q}})^{b,c}$ (see \cref{def:common-area}).
	We consider the two cases~$\Delta^{b,c}(P, Q) = \emptyset$ and~$\Delta^{b,c}(P, Q) \neq \emptyset$.
	In the former case it holds that~$\cross^{b,c}_P(Q) = \{s_P,t_P\}$ and~$\cross^{b,c}_Q(P) = \{s_Q, t_Q\}$ are avoiding.
	In the latter case, there is a~$\delta_P^{b,c}(Q) \neq \bot$ or~$\delta_Q^{b,c}(P) \neq \bot$ by \cref{lem:delta_exists}.
	Without loss of generality, assume that~$\delta_P^{b,c}(Q) \neq \bot$.
	Then,~${\delta_P^{b,c}(Q) \in \cross^{b,c}_P(Q) \subseteq P'}$ and by \cref{lem:delta_disjointness} we have that~$\cross^{b,c}_P(Q)$ and~$\cross^{b,c}_Q(P)$ are avoiding.
\end{proof}

The next lemma deals with the case where~$\alpha^{b,c}_P(Q) \neq \bot$ (i.\,e.\ case (ii) of \cref{fig:areas}).

\begin{lemma}\label{lem:outside-means-avoiding}
	Let~$P$ be a~$b$-colored~\mbox{$s_P$-$t_P$-path}, let~$Q$ be a~$c$-colored~\mbox{$s_Q$-$t_Q$-path}, let~${\alpha^{b,c}_P(Q) \neq \bot}$,
	and let~$S_1$ and~$S_2$ be the segments of~$\cross^{b,c}_P(Q)$ with
	\begin{align*}
	\start(S_1) = s_P,\ \en(S_1) = \alpha_P^{b,c}(Q),\ \start(S_2) = \omega_P^{b,c}(Q),\text{ and }\en(S_2)=t_P.
	\end{align*}
	Then,~$S_1$ and~$\cross^{b,c}_Q(P)$ are avoiding, and so are~$S_2$ and~$\cross^{b,c}_Q(P)$.
\end{lemma}
\begin{proof}
	By \cref{lem:four-diamonds}, the sets
	\[\left(\patharea{\pos{s_P}}{\pos{\partial^{b,c}_P(Q)}}\right)^{b,c}\hspace*{-.1cm}, \left(\patharea{\pos{\varpi^{b,c}_P(Q)}}{\pos{t_P}}\right)^{b,c}\hspace*{-.1cm}, \left(\patharea{\pos{s_Q}}{\pos{\partial^{b,c}_Q(P)}}\right)^{b,c}\hspace*{-.1cm}, \text{\;and } \left(\patharea{\pos{\varpi^{b,c}_Q(P)}}{\pos{t_Q}}\right)^{b,c} \]
	are pairwise disjoint (or undefined).
	Moreover, by \cref{thm:crossing_colored}, it follows that
	\[\NN^k \cap (\patharea{\pos{\alpha_P^{b,c}(Q)}}{\pos{\omega_P^{b,c}(Q)}}) =^{b,c} (\patharea{\pos{\alpha_Q^{b,c}(P)}}{\pos{\omega_Q^{b,c}(P)}}) \cap \NN^k.\]
	As~$P$ is $b$-colored, it follows that~$\left(\patharea{\pos{s_P}}{\pos{\partial^{b,c}_P(Q)}}\right)^{b,c}$ and~$\left(\patharea{\pos{\varpi^{b,c}_P(Q)}}{\pos{t_P}}\right)^{b,c}$ are disjoint from any segment in~$\cross^{b,c}_Q(P)$.
	
	It remains to show that the two minimal segments~$\{\partial^{b,c}_P(Q),\alpha^{b,c}_P(Q)\}$ and~$\{\omega^{b,c}_P(Q),\varpi^{b,c}_P(Q)\}$ avoid~$\cross^{b,c}_Q(P)$.
	To this end, observe that~$\alpha_Q^{b,c}(P),\omega_Q^{b,c}(P) \in \cross^{b,c}_Q(P)$. 
	Moreover, as established above, 
	\[\pos{\partial^{b,c}_P(Q)}^{b,c} \notin \left(\patharea{\pos{s_Q}}{\pos{\partial^{b,c}_Q(P)}}\right)^{b,c} \cup \left(\patharea{\pos{\varpi^{b,c}_Q(P)}}{\pos{t_Q}}\right)^{b,c}\]
	and
	\[\pos{\partial^{b,c}_P(Q)}^{b,c} \notin \left(\patharea{\pos{\alpha_P^{b,c}(Q)}}{\pos{\omega_P^{b,c}(Q)}}\right)^{b,c} = \left(\patharea{\pos{\alpha_Q^{b,c}(P)}}{\pos{\omega_Q^{b,c}(P)}}\right)^{b,c}.\]
	Hence, $\{\partial^{b,c}_P(Q),\alpha^{b,c}_P(Q)\}$ and~$\cross^{b,c}_Q(P)$ are avoiding.
	The case for~$\{\omega^{b,c}_P(Q),\varpi^{b,c}_P(Q)\}$ follows analogously.
\end{proof}

\subsection{Crossing Set}

We now give some intuition on how to lift our ideas to arbitrary~$k$ leading to \cref{def:crossing-sets}, a generalization of \cref{def:2d-crossing-set}.

Initially, each path~$P_i$ gets label~$i$.
Intuitively, we will set a label~$j$ to a (sub)path whenever we will in future exploit that the subpath is~$j$-colored.
For example, if a subpath of~$P_i$ is shown to be geometrically separated from all other subpaths, then we do not need the information that this subpath has a color~$j \neq i$.

Whenever two paths~$P_i$ and~$P_j$ in the solution intersect in the~\mbox{$(i,j)$-projection} (that is, the respective~$\alpha$- and~$\omega$-vertices are not~$\bot$), then the subpaths~$P_i'$ and~$P_j'$ in the intersection get both labels~$i$ and~$j$. 
If a third path~$P'$ also intersects with~$P'_j$, then we try to use the intersections to move the label~$i$ via path~$P_j$ to some subpath of~$P'$.
Generalizing this, we consider for each sequence~$\sigma = (\ell_1, \ell_2, \ldots, \ell_h)$ whether label~$\ell_1$ could be ``transported'' from~$P_{\ell_1}$ to~$P_{\ell_2}$, from~$P_{\ell_2}$ to~$P_{\ell_3}$, and so on until~$P_{\ell_{h}}$.

Keeping track of these sequences allows us to bound the number of resulting subpaths:
We will show that for each such sequence~$\sigma = (\ell_1, \ell_2, \ldots, \ell_h)$ at most one subpath of~$P_{\ell_h}$ can receive label~$\ell_1$ via~$\sigma$.

In the following, we use~$\setp(\tau) \defeq \{\ell_1, \ldots, \ell_h\}$ to denote the set with all entries in a sequence~$\tau = (\ell_1, \ldots, \ell_h)$.
We next define the \emph{crossing set}~$\crossing$ recursively for each~$\Phi \subseteq [k]$.
This should be seen as the set of marbles of a solution.
We will then show a result similar to \cref{lem:no-alpha-means-avoiding,lem:outside-means-avoiding} for arbitrary~$k$ that then allows us to find the desired partition of paths.

\begin{definition}\label{def:crossing-sets}
	Let~$(G,(s_i,t_i)_{i\in [k]})$ be an instance of \kDSP{} and let~$\calP = (P_i)_{i \in [k]}$ be a solution to this instance, that is, $P_i$ is the path between~$s_i$ and~$t_i$ in the solution.
	For each~$\Phi \subseteq [k]$ and each permutation~${\sigma = (\ell_1, \ldots, \ell_{|\Phi|})}$ of~$\Phi$, we define the crossing set~$\crossing^{\sigma}$ and the endpoints~$\segmentEnds(\sigma)$ of intersections as follows.
	\begin{itemize}
		\item If $|\Phi|=1$ with $\sigma = (i)$, then let $\crossing^{\sigma} \defeq\segmentEnds(\sigma) \defeq \{s_i, t_i\}$.
		\item If $|\Phi|=2$ with $\sigma = (i,j)$, then let
		\[
			\segmentEnds(\sigma) \defeq \{\alpha^{i,j}_{P_j}(P_i), \omega^{i,j}_{P_j}(P_i)\} \text{ and }
			\crossing^{\sigma} \defeq \cross^{i,j}_{P_j}(P_i) \setminus \{\bot\}.
		\]
		\pagebreak[0]
		\item If $|\Phi| \ge 3$, then let~$\sigma_{\start} \defeq (\ell_1, \ldots, \ell_{|\Phi|-1})$ and~$\sigma_{\en} \defeq (\ell_2, \ldots, \ell_{|\Phi|})$. 
			We denote by~$Q$ the maximum common subpath of~$P_{\ell_{|\Phi|-1}}[\segmentEnds(\sigma_{\start})]$ and~$P_{\ell_{|\Phi|-1}}[\segmentEnds((\ell_{|\Phi|},\ell_{|\Phi|-1}))]$.
			If $\segmentEnds(\sigma_{\start}) = \{\bot\}$,\ \ $\segmentEnds(\sigma_{\en}) = \{\bot\}$, or~${Q = \emptyset}$,
			then let~${\segmentEnds(\sigma) \defeq \crossing^{\sigma} \defeq \{\bot\}}.$
			Otherwise, let
			\begin{align*}
				P &\defeq P_{\ell_{|\Phi|}}[\segmentEnds(\sigma_{\en})],\\
				\segmentEnds(\sigma) &\defeq \{\alpha^{\ell_1,\ell_{|\Phi|}}_{P}(Q), \omega^{\ell_1,\ell_{|\Phi|}}_{P}(Q)\}, \text{ and}\\
				\crossing^{\sigma} &\defeq (\cross^{\ell_1, \ell_{|\Phi|}}_P(Q) \cup \cross^{\ell_1, \ell_{|\Phi|}}_Q(P)) \setminus \{\bot\}.
			\end{align*}
	\end{itemize}
	The set~$\crossing \defeq \bigcup_{\sigma} \crossing^\sigma$ is the \emph{crossing set} of $\calP$.
\end{definition}

We make the following observation, which implies that the above definition of crossing sets and endpoints is valid.
\begin{observation}
		\label{thm:crossing-set-valid}
	Let~$\sigma \defeq (\ell_1, \dots, \ell_{|\Phi|})$ be any permutation of any~$\Phi \subseteq [k]$.
	If~$\segmentEnds(\sigma) \ne \{\bot\}$, then
	\begin{enumerate}[(i)]
		\item $\segmentEnds(\sigma) \subseteq P_{\ell_{|\Phi|}}$, and
		\item $\segmentEnds(\sigma)$ is $c$-colored for each~$c \in \Phi$.
	\end{enumerate}
\end{observation}
\begin{proof}
	We prove both claims by an induction over~$|\Phi|$.
	For~$|\Phi| = 1$, note that~${\segmentEnds(\sigma) = \{s_{\ell_1}, t_{\ell_1}\}}$.
	Clearly~$\{s_{\ell_1}, t_{\ell_1}\} \subseteq P_{\ell_1}$ as these are the ends of~$P_{\ell_1}$ and the pair~$\{s_{\ell_1}, t_{\ell_1}\}$ is by definition~$\ell_1$-colored.
	
	Now assume that both claims hold for all~$\Phi'$ with~$|\Phi'| < |\Phi|$.
	Since~${\segmentEnds(\sigma) \ne \{\bot\}}$, it holds that~$\segmentEnds(\sigma) = \{\alpha^{\ell_1,\ell_{|\Phi|}}_{P}(Q), \omega^{\ell_1,\ell_{|\Phi|}}_{P}(Q)\}$, where
	\begin{align*}
		P &\coloneqq P_{\ell_{\Phi}}[\segmentEnds((\ell_2, \dots, \ell_{\abs{\Phi}}))], \\
		Q &\coloneqq P_{\ell_{|\Phi|-1}}[\segmentEnds((\ell_1,\ell_2,\ldots,\ell_{|\Phi|-1}))] \cap P_{\ell_{|\Phi|-1}}[\segmentEnds((\ell_{|\Phi|},\ell_{|\Phi|-1}))]
	\end{align*}
	if~$|\Phi| \geq 3$,
	and if~$|\Phi|=2$, then~$P \coloneqq P_{\ell_2}, Q \coloneqq P_{\ell_1}$.
	Note that~$Q \neq \emptyset$ and hence
	by induction hypothesis $Q$~is~$c$-colored for each~${c \in \Phi \setminus \{\ell_{|\Phi|}\}}$.
	Thus,
	\[{\segmentEnds(\sigma) =  \{\alpha^{\ell_1,\ell_{|\Phi|}}_{P}(Q), \omega^{\ell_1,\ell_{|\Phi|}}_{P}(Q)\}}\]
	is well-defined and hence~$\segmentEnds(\sigma) \subseteq P \subseteq P_{\ell_{|\Phi|}}$.
	Moreover, by \cref{thm:crossing_colored}, it holds that~$\segmentEnds(\sigma)$ is~\mbox{$c$-colored} for each~${c \in (\Phi \setminus \{\ell_{|\Phi|}\}) \cup \{\ell_{|\Phi|}\} = \Phi}$.
\end{proof}

We next want to show a lemma (\cref{lem:parallel-case} which states that when ``transporting'' the labels via a permutation~${\sigma = (\ell_1, \ell_2, \ldots, \ell_{|\Phi|})}$,
then the intersecting subpath in the target path~$P_{\ell_{|\Phi|}}$ ``agrees'' in all coordinates in~$\setp(\sigma)$ with a subpath of the path~$P_{\ell_{|\Phi|-1}}$ the label is transported from.
The proof uses the following observation:
For any sequence~${\sigma' = (\ell_i,\ell_{i+1},\ldots,\ell_{|\Phi|})}$ with~$i \geq 1$,
it holds that the subpath of~$P_{\ell_{|\Phi|}}$ between the two vertices in~$\segmentEnds(\sigma)$ is a subpath of the one between the two vertices in~$\segmentEnds(\sigma')$.
In other words, if we add more entries to the front of~$\sigma'$, then we get smaller and smaller paths.

\begin{observation}
		\label{obs:less-is-more}
	Let $\sigma \defeq (\ell_1, \ell_2, \dots, \ell_{|\Phi|})$ be any permutation of any $\Phi \subseteq [k]$ with~${\abs{\Phi} \geq 2}$.
	If~$\segmentEnds(\sigma) \ne \{\bot\}$, then
	$P_{\ell_{|\Phi|}}[\segmentEnds(\sigma)] \subseteq P_{\ell_{|\Phi|}}[\segmentEnds((\ell_2, \ell_3, \dots, \ell_{|\Phi|}))].$
\end{observation}
\begin{proof}
	Follows directly from~\cref{def:crossing-sets} as the vertices~$\segmentEnds(\sigma)$ lie on~$P_{\ell_{\abs{\Phi}}}[\segmentEnds((\ell_2, \ell_3, \ldots, \ell_{\abs{\Phi}}))]$.
\end{proof}

With this observation, we are able to show the following lemma.

\begin{lemma}
	\label{lem:parallel-case}
	Let~$\Phi \subseteq [k]$ with~$|\Phi| \geq 2$.
	Let $\sigma \defeq (\ell_1, \ell_2, \dots, \ell_{|\Phi|})$ be any permutation of~$\Phi$.
	If~$\segmentEnds(\sigma) \ne \{\bot\}$, then
	$P_{\ell_{|\Phi|}}[\segmentEnds(\sigma)] =^{\Phi} Q'$ for some subpath $Q'$ of
	\[
		{Q \defeq P_{\ell_{|\Phi|-1}}[\segmentEnds((\ell_1, \ell_2, \dots, \ell_{|\Phi|-1}))] \cap P_{\ell_{|\Phi|-1}}[\segmentEnds((\ell_{|\Phi|}, \ell_{|\Phi|-1}))]}.
	\]
\end{lemma}

\begin{proof}
	Note that the existence of $Q' \subseteq Q$
	with $P_{\ell_{\abs{\Phi}}}[\segmentEnds(\sigma)] =^{\ell_1, \ell_{\abs{\Phi}}} Q'$
	follows directly from \cref{def:crossing-sets}
	and~\cref{thm:crossing_colored}.
	In particular, the claim holds for $\abs{\Phi}=2$.
	We will use induction over~$|\Phi|$ to also prove the claim for $\abs{\Phi} > 2$.
	Let~$\sigma_{\start} \defeq (\ell_1,\ell_2,\ldots,\ell_{|\Phi|-1})$ and~${\sigma_{\en} \defeq (\ell_2, \ell_3, \dots, \ell_{|\Phi|})}$.
	By \cref{obs:less-is-more},~$P_{\ell_{|\Phi|}}[\segmentEnds(\sigma)] \subseteq P_{\ell_{|\Phi|}}[\segmentEnds(\sigma_{\en})]$
	and hence there is by induction hypothesis a subpath
	\[
		R' \subseteq P_{\ell_{|\Phi|-1}}[\segmentEnds((\ell_2, \dots, \ell_{\abs{\Phi}-1}))] \cap P_{\ell_{|\Phi|-1}}[\segmentEnds((\ell_{|\Phi|}, \ell_{|\Phi|-1}))]
	\]
	with $P_{\ell_{|\Phi|}}[\segmentEnds(\sigma_{\en})] =^{\setp(\sigma_{\en})} R'$.
	Hence, $R' \supseteq^{\ell_{|\Phi|}} P_{\ell_{|\Phi|}}[\segmentEnds(\sigma)] =^{\ell_{|\Phi|}} Q'$ and~$R'$ and~$Q'$ are both subpaths of~$Q$.
	Since~$Q, R' \subseteq P_{\ell_{|\Phi|-1}}[\segmentEnds((\ell_{|\Phi|}, \ell_{|\Phi|-1}))]$ are by \cref{thm:crossing-set-valid} \mbox{$\ell_{|\Phi|}$-colored}, this implies~$R' \supseteq Q'$.
	From~$P_{\ell_{\abs{\Phi}}}[\segmentEnds(\sigma)] =^{\ell_1, \ell_{\abs{\Phi}}} Q'$ follows~$Q' =^{\setp(\sigma)} P_{\ell_{|\Phi|}}[\segmentEnds(\sigma)]$, which proves the claim.
\end{proof}

\subsection{Labels and Avoiding Marble Paths}

In this subsection, we finally state the central result of this section.
To this end, we first define labels for an~$i$-marble path~$T$
\[
\varlabels[T] := \{a \mid \exists \sigma = (a=\ell_1,\ldots,\ell_h=i), h\ge 1 \colon T \subseteq P_i[\segmentEnds(\sigma)]  \}
\]
In our algorithm, the marble paths $P_i \cap \crossing$ play a central role.
Their significance is due to the fact that if the labels of two minimal segments are not identical, then the two segments are avoiding.
This result will be proven in \cref{lem:correctness} at the end of this section.
Before we can formally prove it, we will first generalize \cref{lem:no-alpha-means-avoiding,lem:outside-means-avoiding} from~$\cross$ (comparison of two paths) to~$\crossing$ (sequences of paths).
Unfortunately, the following proof contains a lot of rather tedious case distinctions.

\begin{lemma}
		\label{lem:avoiding-cases}
	Let~$(G,(s_i,t_i)_{i\in [k]})$ be an instance of \kDSP,~$\calP \defeq (P_i)_{i \in [k]}$ be a solution to this instance,~$\crossing$ be the crossing set of~$\calP$,~$\Phi \subseteq [k]$ be some set,~${\sigma \defeq (\ell_1, \ell_2, \ldots, \ell_{|\Phi|})}$ be a permutation of~$\Phi$,~${\sigma_{\start} \defeq (\ell_1, \ell_2, \ldots, \ell_{|\Phi|-1})}$,~$g\defeq \ell_1, i \defeq \ell_{|\Phi|-1}$, and let~$j \defeq \ell_{|\Phi|}$.
	\begin{enumerate}[(i)]
		\item If $\segmentEnds(\sigma) = \{\bot\}$ and $\segmentEnds(\sigma_{\start}) \ne \{\bot\}$, then the two \mpath{s}
			\[ P_{i}[\segmentEnds(\sigma_{\start})] \cap \crossing \quad\text{and}\quad P_{j} \cap \crossing \]
			are avoiding. 
		\item If $\segmentEnds(\sigma) = \{u,v\}$ with~$u <^j v$, then the two \mpath{s}
			\[ P_{i}[\segmentEnds(\sigma_{\start})] \cap \crossing \quad\text{and}\quad P_{j}[s_j, u] \cap \crossing\]
			are avoiding and so are
			\[ P_{i}[\segmentEnds(\sigma_{\start})] \cap \crossing \quad\text{and}\quad P_{j}[v, t_j] \cap \crossing.\]
	\end{enumerate}
\end{lemma}
\begin{proof}
	We will prove both claims simultaneously by induction over $|\Phi|$.

	\noindent
	\emph{Base case:} Let~$|\Phi|=2$ and hence $g=i$ and~$P_i[\segmentEnds(\sigma_{\start})] = P_i$.
	\begin{enumerate}[(i)]
		\item 
			\label{claim1}
			Since~$\segmentEnds(\sigma) = \{\bot\}$, it holds that~$\alpha^{i,j}_{P_j}(P_i) = \bot$.
			By \cref{lem:no-alpha-means-avoiding}, $\cross^{i,j}_{P_i}(P_j)$ and $\cross^{i,j}_{P_j}(P_i)$ are avoiding.
			Hence, $P_{i} \cap \crossing$ and~$P_{j} \cap \crossing$ are also avoiding by \cref{thm:supersegments-avoiding}
			since
			${\cross^{i,j}_{P_i}(P_j) \subseteq P_i \cap \crossing}$ and ${\cross^{i,j}_{P_j}(P_i) \subseteq P_j \cap \crossing}$
			by \cref{def:crossing-sets}.
			
		\item
			\label{claim2}
			Since $\sigma = (i,j)$, it holds that~$u = \alpha^{i,j}_{P_j}(P_i)$ and $v = \omega^{i,j}_{P_j}(P_i)$.
			By \cref{lem:outside-means-avoiding}, $\cross^{i,j}_{P_i}(P_j)$ and $\{s_j, u\}$ are avoiding and so are~$\cross^{i,j}_{P_i}(P_j)$ and $\{v, t_j\}$.
			Thus the claim again follows from \cref{thm:supersegments-avoiding} 
			since ${\cross^{i,j}_{P_i}(P_j) \subseteq P_i \cap \crossing}$ and~${\cross^{i,j}_{P_j}(P_i) \subseteq P_j \cap \crossing}$.
	\end{enumerate}
	
	\noindent
	\emph{Induction step:} Let~$|\Phi|\geq 3$ and assume that both claims holds for all~${\Phi' \subseteq [k]}$ with~${2 \leq |\Phi'| < |\Phi|}$.
	Let $\sigma_{\en} \defeq (\ell_2, \ell_3, \ldots, \ell_{|\Phi|})$ and~$\sigma' \defeq (\ell_2, \ell_3, \ldots,\ell_{|\Phi|-1})$.
	In the following, we will frequently use \cref{thm:supersegments-avoiding} implicitly.
	\begin{enumerate}[(i)]
		\item Since~$\segmentEnds(\sigma) = \{\bot\}$ and~$\segmentEnds(\sigma_{\start}) \neq \{\bot\}$, by \cref{def:crossing-sets}, there are three possible cases:
		\begin{align*}
		\segmentEnds(\sigma_{\en}) &= \{\bot\},\\
		Q \defeq P_{i}[\segmentEnds(\sigma_{\start})] \cap P_{i}[\segmentEnds((j,i))] &= \emptyset, \text{ or }\\
		\alpha^{g,j}_{P_j[\segmentEnds(\sigma_{\en})]}(Q) &= \bot.
		\end{align*}
		We will show that~$P_{i}[\segmentEnds(\sigma_{\start})] \cap \crossing \text{ and } P_{j} \cap \crossing$ are avoiding in each of the three cases.
		\begin{enumerate}[(1)]
			\item We start with the case where~$\segmentEnds(\sigma_{\en}) = \{\bot\}$.
				Since~$\segmentEnds(\sigma_{\start}) \neq \{\bot\}$ it holds by~\cref{obs:less-is-more} that
				\[\emptyset \neq P_i[\segmentEnds(\sigma_{\start})] \subseteq P_i[\segmentEnds(\sigma')]\]
				and in particular,~$\segmentEnds(\sigma') \neq \{\bot\}$.
				Since also~$\segmentEnds(\sigma_{\en}) = \{\bot\}$, the induction hypothesis~\ref{claim1} states that
				$P_{i}[\segmentEnds(\sigma')] \cap \crossing$ and $P_{j} \cap \crossing$
				are avoiding.
				By definition, each segment of~$P_{i}[\segmentEnds(\sigma')] \cap \crossing$ is also avoiding with~$P_{j} \cap \crossing$
				and since~$P_i[\segmentEnds(\sigma_{\start})] \subseteq P_i[\segmentEnds(\sigma')]$ and~$\segmentEnds(\sigma_{\start}) \subseteq \crossing$, it holds that~$P_{i}[\segmentEnds(\sigma_{\start})]\cap \crossing$ is a segment of~$P_{i}[\segmentEnds(\sigma')] \cap \crossing$.
			\item We continue with the case where
				\[Q \defeq P_i[\segmentEnds(\sigma_{\start})] \cap P_i[\segmentEnds((j,i))] = \emptyset.\]
				First, if $\segmentEnds((j,i)) = \{\bot\}$, then the base case~\ref{claim1} gives that 
				$P_i \cap \crossing$ and $P_j \cap \crossing$ are avoiding,
				so we are done since $P_i[\segmentEnds(\sigma_{\start})] \cap \crossing$ is a segment of $P_i \cap \crossing$.
				
				So assume now $\segmentEnds((j,i)) \neq \{\bot\}$.
				Then, by base case~\ref{claim2},
				$P_j \cap \crossing$ is avoiding with both
				\[P_i[s_i, \alpha^{i,j}_{P_i}(P_j)] \cap \crossing \quad\text{and}\quad P_i[\omega^{i,j}_{P_i}(P_j), t_i] \cap \crossing.\]
				Since~$Q = \emptyset$, it holds that
				$$P_i[\segmentEnds(\sigma_{\start})] \subseteq P_i[s_i, \alpha^{i,j}_{P_i}(P_j)] \text{ or } P_i[\segmentEnds(\sigma_{\start})] \subseteq P_i[\omega^{i,j}_{P_i}(P_j),t_j].$$
				Assume \wilog{} the former.
				Then $P_i[\segmentEnds(\sigma_{\start})] \cap \crossing$
				is a segment of~$P_i[s_i, \alpha^{i,j}_{P_i}(P_j)] \cap \crossing$.
				Since we also have by \cref{lem:outside-means-avoiding} that
				$P_i[s_i, \alpha^{i,j}_{P_i}(P_j)] \cap \crossing$ and~$P_j \cap \crossing$ are avoiding,
				it follows that~$P_i[\segmentEnds(\sigma_{\start})] \cap \crossing$ and~$P_j \cap \crossing$ are avoiding.
			\item It remains to analyze the case where~$\alpha^{g,j}_{P_j[\segmentEnds(\sigma_{\en})]}(Q) = \bot$.
				We assume~${\segmentEnds(\sigma_{\en}) \neq \{\bot\}}$ and~${Q \neq \emptyset}$ as we can otherwise use the proofs above.
				Assume towards a contradiction that~${P_j \cap \crossing}$ and~$P_i[\segmentEnds(\sigma_{\start})] \cap \crossing$ are not avoiding.
				Then there are minimal segments~$S_i \subseteq P_i[\segmentEnds(\sigma_{\start})] \cap \crossing$ and~$S_j \subseteq P_j \cap \crossing$ that are not avoiding.
				
				We claim that $S_j \subseteq P_j[\segmentEnds(\sigma_{\en})]$,
				since otherwise the induction hypothesis~\ref{claim2} (applied to $\sigma_{\en}$) would state that
				$S_j$ is avoiding with $P_i[\segmentEnds(\sigma')] \cap \crossing$.
				This would contradict our assumption since $P_i[\segmentEnds(\sigma_{\start})] \cap \crossing$ is a segment of $P_i[\segmentEnds(\sigma')] \cap \crossing$.
				
				By \cref{lem:outside-means-avoiding}, we further have $S_i \subseteq Q$.
				But since $\alpha^{g,j}_{P_j[\segmentEnds(\sigma_{\en})]}(Q) = \bot$,
				\cref{lem:no-alpha-means-avoiding} yields that $P_j[\segmentEnds(\sigma_{\en})] \cap \crossing$ and $Q \cap \crossing$ are avoiding --- a contradiction.
		\end{enumerate}
	\item
		Since the two claims in~\ref{claim2} are symmetrical, we will only show that~${P_{i}[\segmentEnds(\sigma_{\start})] \cap \crossing}$ and~$P_{j}[s_j, u] \cap \crossing$ are avoiding.
		To this end, assume towards a contradiction that there are minimal non-avoiding segments
		\[ S_i \subseteq P_i[\segmentEnds(\sigma_{\start})] \cap \crossing \quad\text{and}\quad S_j \subseteq P_j[s_j,u] \cap \crossing. \]
		
		We assume that $S_i \subseteq P_i[\segmentEnds((j, i))]$ and thus $S_i \subseteq Q$.
		Otherwise, $S_i$ and $P_j \cap \crossing$ are avoiding
		by \cref{lem:outside-means-avoiding} (if $\segmentEnds((j, i)) \neq \{\bot\}$)
		or by \cref{lem:no-alpha-means-avoiding} (if $\segmentEnds((j, i)) = \{\bot\}$),
		and~$P_j[s_j,u] \cap \crossing$ is clearly a segment of $P_j \cap \crossing$.
		
		By assumption, we have $\segmentEnds(\sigma) \neq \{\bot\}$ implying~$\segmentEnds(\sigma_{\en}) \neq \{\bot\}$ by \cref{obs:less-is-more}.
		Let~$\{y,z\} \defeq \segmentEnds(\sigma_{\en})$  with $y <^j z$.
		If~$S_j \subseteq P_j[y,u] \cap \crossing$, then~${S_i \subseteq Q \cap \crossing}$
		and $S_j$ are avoiding by \cref{lem:outside-means-avoiding}.
		
		Otherwise, we have $S_j \subseteq P_j[s_j, y]$.
		Applying induction hypothesis \ref{claim2} to $\sigma_{\en}$ then gives that~${S_i \subseteq P_i[\segmentEnds(\sigma')] \cap \crossing}$ and $S_j \subseteq P_j[s_j,y] \cap \crossing$
		are avoiding.
		\qedhere
	\end{enumerate}
\end{proof}

We are now equipped to show the main result of this section.

\begin{proposition}
	\label{lem:correctness}
	For~$i,j \in [k]$ let $S_i\subseteq P_i \cap \crossing$ and $S_j \subseteq P_j \cap \crossing$ be two minimal segments.
	If~$\varlabels[S_i] \ne \varlabels[S_j]$, then~$S_i$ and~$S_j$ are avoiding.
\end{proposition}

\begin{proof}
	We start with the case where~$i \notin \varlabels[S_j]$.
	Then by \cref{def:crossing-sets},~${\segmentEnds((i,j)) = \{\bot\}}$
	or~$S_j$ is not a segment of $P_j[u,v] \cap \crossing$, where~$\{u,v\} \defeq \segmentEnds((i,j)) \neq \{\bot\}$.
	In both cases $S_j$~and~$P_i \cap \crossing$ are avoiding by applying \cref{lem:avoiding-cases} to $\sigma = (i,j)$.
	The case where~$j \notin \varlabels[S_i]$ is analogous.
	
	It remains to consider the case where
	$i,j \in \varlabels[S_i] \cap \varlabels[S_j]$.
	\Wilog{} there is~$d \in \varlabels[S_i] \setminus \varlabels[S_j]$.
	By \cref{def:crossing-sets}, there is a set~${\Phi = \{\ell_1, \ell_2, \ldots, \ell_{|\Phi|}\}}$ and a permutation~${\sigma = (\ell_1, \ell_2, \ldots, \ell_{|\Phi|})}$ of~$\Phi$ such that $\ell_1 = d$,~$\ell_{|\Phi|} = i$, and $S_i \subseteq P_i[\segmentEnds(\sigma)] \cap \crossing$.
	
	If~$j \notin \Phi$, then apply \cref{lem:avoiding-cases} to~$\sigma' \defeq (\ell_1,\ell_2,\ldots,\ell_{|\Phi|},j)$
	to see that $S_j$ and $P_i[\segmentEnds(\sigma)]$ are avoiding
	(note that $S_j$ cannot be a segment of~$P_j[\segmentEnds(\sigma')]$	since~$d \notin \varlabels[S_j]$).
	
	Now assume that~$j \in \Phi$
	and let $x$ be such that~$j = \ell_x$.
	For all $h \leq \abs{\Phi}$, we denote by~${\sigma_h \defeq (\ell_1, \ell_2, \ldots, \ell_h)}$ the respective prefix of~$\sigma$.
	Next, since~$S_i \subseteq P_i[\segmentEnds(\sigma)]  \cap \crossing$, it follows that~${\segmentEnds(\sigma) \neq \{\bot\}}$
	and, by \cref{def:crossing-sets}, $\segmentEnds(\sigma_h) \neq \{\bot\}$ for all~$h$.
	By \cref{lem:parallel-case}, for each $h > x$ there is
	a subpath~$Q'$ of $P_{\ell_{h-1}}$
	such that $P_{\ell_h}[\segmentEnds(\sigma_h)] =^{\setp(\sigma_h)} Q'$.
	Thus, by transitivity, there is a subpath~$Q''$ of~$P_j[\segmentEnds(\sigma_x)]$ with $Q'' =^{\setp(\sigma_x)} P_i[\segmentEnds(\sigma)]$.
	In particular, $S_i \subseteq^{j} P_j[\segmentEnds(\sigma_x)]$.

	Let $Q_i$ be a path following~$S_i$ and $Q_j$ a path following $S_j$.
	By \cref{thm:crossing-set-valid}, $Q_i$ and~$Q_j$ are both $j$-colored.
	By the above $Q_i \subseteq^j P_j[\segmentEnds(\sigma_x)]$.
	Since $d \notin \varlabels[S_j]$, it follows that~$S_j$ is not a segment of $P_j[\segmentEnds(\sigma_x)] \cap \crossing$ and therefore~$\pos{Q_j}^j \cap \pos{P_j[\segmentEnds(\sigma_x)]}^j \subseteq \pos{\segmentEnds(\sigma_x)}^j$.
	
	This proves~$\pos{Q_i}^j \cap \pos{Q_j}^j \subseteq \pos{P_j[\segmentEnds(\sigma_x)]}^j$
	and thus $S_i$ and $S_j$ are $\{j\}$-avoiding.
\end{proof}

\section{The Algorithm: Utilizing the Geometry}\label{sec:the-algorithm}
\label{sec:algorithm}

In this section, we finally present the algorithm behind \cref{thm:main-thm}.
In a nutshell, we first guess all the \mpath{s}~$T_i := P_i \cap \crossing$ and the respective ends~$\segmentEnds$ corresponding to some solution (if one exists).
We then compute all minimal segments of each \mpath{}~$T_i$, compute their respective~$\varlabels$,
and partition the segments such that all segments in the same part of the partition are strictly monotone in a common coordinate 
and that segments in distinct parts of the partition are avoiding.
The crucial improvement over the algorithm by \citet{Loc21} is that our partition is much smaller.
Afterwards, we find for all segments in one part of the partition disjoint paths that follow the respective segments via dynamic programming.

To this end, we look at directed acyclic graphs (DAGs) that result from the input graph by removing and directing the edges.
For a graph~$G$ with vectors~$\pos{v}$ for all~${v\in V}$ (as defined in \cref{sec:2D-geometry}) and a coordinate~$c$, we define the DAG~${D_c=(V,A)}$ with
\[A = \{(x,y) \mid \{x,y\} \in E(G) \land~\pos{y}^c - \pos{x}^c = 1\}.\]
Note that each edge~$\{u,v\}$ in any~$c$-colored path~$P$ fulfills~$|\pos{u}^c - \pos{v}^c| = 1$ by definition and each path in~$D_c$ corresponds to a~$c$-colored path in~$G$.
Moreover, all paths between two vertices~$u$ and~$v$ in~$D_c$ have the same length as they each need to contain exactly one vertex for each~$c$-coordinate~$i \in [\pos{u}^c,\pos{v}^c]$.
Thus, the following observation holds true.
\begin{observation}
	\label{obs:path-dag}
	A path $P$ in $G$ is $c$-colored if and only if $P$ or its reversal is a path in~$D_c$.
\end{observation}
Next, we define the problem \DPD.
We are given a DAG~$D$ and a list $(s_i,t_i)_{i\in [p]}$ of (possibly intersecting) terminal pairs, 
and ask whether there are pairwise internally disjoint $s_i$-$t_i$-path in $D$, for each $i \in [p]$.
Formally:
\decProb{\DPD}{A DAG~$D$ and~$p$ pairs~$(s_i,t_i)_{i \in [k]}$ of vertices.}{Are there~$p$ internally vertex-disjoint paths~$P_i$ in~$D$ such that each~$P_i$ is a shortest~\mbox{$s_i$-$t_i$-path}?}

\Citet{FHW80} showed an $n^{O(p)}$-time algorithm for \textsc{$p$-Disjoint Path on DAGs}.
For the completeness of our algorithm and to drop the big-O in the exponent, we 
show that~\textsc{\mbox{$p$-Disjoint} Paths on DAGs} can be solved in $O(n^{p-1}m + pn^p)$ time.
Our algorithm for~\textsc{\mbox{$p$-Disjoint} Paths on DAGs} does not contain new algorithmic ideas 
compared to the algorithm of \mbox{\citet{FHW80}}.

\begin{lemma}
		\label{lem:DP}
	An instance of \textsc{$p$-Disjoint Paths on DAGs} on a graph with~$n$~vertices can be solved in~$O(n^{p-1}m + pn^p)$ time.
\end{lemma}

\begin{proof}
Let~$D = (V,A)$ be a DAG and let~$(s_i,t_i)_{i \in [p]}$ be a set of~$p$ terminal pairs.
We define~$V^{\en} \defeq \bigcup_{i \in [p]}\{s_i,t_i\}$ to be the set of all terminals.
We also choose an arbitrary topological order of~$D$ and denote by~$u \prec v$ that~$u$ comes before~$v$ in this topological order.
We assume without loss of generality that~$s_i \preceq s_j$ for all~$i < j \in [p]$.
We further assume that~$s_i \preceq t_i$ for all~$i \in [p]$ as otherwise there can be no path from~$s_i$ to~$t_i$ and that~$p \leq n$ as we can iterate over all pairs~$(s_i,t_i)$ and delete those that are connected by an arc~$(s_i,t_i) \in A$.
All remaining paths have at least one inner vertex that has to be from~$V \setminus V^{\en}$ and that has to be unique for each path.
Hence, if there are at least~$n+1$ pairs remaining, then the instance has no solution.

We build a table~$T[x_1,x_2,\ldots,x_p] \in \{\true,\false\}$ that stores~$\true$ if and only if the following three criteria are fulfilled.
\begin{enumerate}[(i)]
	\item $x_i \in (V \setminus V^{\en}) \cup \{s_i,t_i\}$,
	\item $s_i \preceq x_i \preceq t_i$ for all~$x_i \in \{s_i,t_i\}$, and
	\item there exist $s_i$-$x_i$-paths such that each inner vertex of each of these paths is in~$V \setminus V^{\en}$ and that each vertex in~$V \setminus V^{\en}$ is contained in at most one of these paths. 
\end{enumerate}

If the table is completely filled, then there is a set of internally vertex-disjoint shortest~\mbox{$s_j$-$t_j$-paths} if and only if~$T[t_1,t_2,\ldots,t_p] = \true$ as the first two requirements are trivially fulfilled.
We initialize the table with~${T[s_1,s_2,\ldots,s_p] \defeq \true}$ as internally vertex-disjoint~\mbox{$s_i$-$s_i$-paths} trivially exist.
Moreover, for each tuple~$(x_1,\ldots,x_p) \in V^p$ if~${x_i \prec s_i}$,~${t_i \prec x_i}$, or~$x_i \in V^{\en} \setminus \{s_i,t_i\}$ for at least one~$i \in [p]$ or~$x_i = x_j \in V \setminus V^{\en}$ for some~$i \neq j$, then we set~$T[x_1,\ldots,x_p] \defeq \false$.
Note that there are~$n^p$ possible tuples and the initialization can be done in~$O(1)$ time per entry.

We next show how to compute the entries of~$T$.
To this end, for some tuple~$(x_1,x_2,\ldots,x_p)$, let~$x_\ell$ be a vertex such that~$x_\ell \neq s_\ell$ and~$x_i \preceq x_\ell$ for all~$x_i$ with~$i \in [p]$ and~$x_i \neq s_i$.
Moreover, let
\[N^*(x_i) \defeq \{v \mid (v,x_i) \in A \land v \in  \{s_i\} \cup (V \setminus (V^{\en} \cup \{x_1, \ldots, x_p\})) \}.\]
Finally, let
\[T[x_1,x_2,\ldots,x_p] \defeq \bigvee_{x'_\ell \in N^*(x_\ell)} T[x_1,x_2,\ldots,x_{\ell-1},x'_\ell,x_{\ell+1},\ldots, x_p].\]

We now show by induction on the sum of positions in the topological order of all~$x_i$ that~$T[x_1,x_2,\ldots,x_p] = \true$ if and only if the three criteria are fulfilled.
In the base case,~$x_i = s_i$ for all~$i \in [p]$ and hence~$T[x_1,x_2,\ldots,x_p]=\true$ or there is some~$x_i$ such that~$x_i \prec s_i$ and therefore~$T[x_1,x_2,\ldots,x_p]=\false$.
Note that there is no~$s_i$-$x_i$-path in the latter case.

Now to show the statement for some table entry~$T[x_1,x_2,\ldots,x_p]$, assume that the statement holds for all table entries~$T[x'_1,x'_2,\ldots,x'_p]$ such that~$x'_i \preceq x_i$ for all~$i \in [p]$ and~$x'_j \prec x_j$ for at least one~$j \in [p]$.
To this end, first assume that~$T[x_1,x_2,\ldots,x_p] = \true$.
Since~${T[x_1,x_2,\ldots,x_p] = \true}$, it was not set to~$\false$ in the initialization and thus i) and ii) are satisfied.
By construction, there is an~$x'_\ell \in N^*(x_\ell)$ such that \[{T[x_1,x_2,\ldots,x_{\ell-1},x'_\ell,x_{\ell+1},\ldots,x_p] = \true}.\]
By the induction hypothesis, there are internally vertex-disjoint~$s_\ell$-$x'_\ell$- and~\mbox{$s_j$-$x_j$-paths} for all~${j \in [p] \setminus \{\ell\}}$ such that~$s_\ell \preceq x'_\ell \preceq t_\ell$ and~$x'_\ell \in (V \setminus V^{\en}) \cup \{s_\ell,t_\ell\}$.
By definition of~$x_\ell$ it holds that~$x_i \preceq x_\ell$ for all~$x_i$ with~$i\in [p]$ and~$x_i \neq s_i$.
Moreover, by the recursive definition of~$T$ it holds that~$x_\ell$ is not contained in any of the~$s_i$-$x_i$-paths for~$i \in [p]$ or~$x_\ell = t_\ell$.
Note that in the latter case it holds for each~$i \in [p]$ that if~$x_i = x_\ell$, then~$x_i \in \{s_i,t_i\}$ by induction hypothesis.
Thus, the~$s_\ell$-$x'_\ell$-path can be extended by the arc~$(x'_\ell,x_\ell)$ and the resulting path combined with the other~$s_i$-$x_i$-paths satisfies iii) in both cases.

To show the other direction assume that~$x_1,x_2,\ldots,x_p$ satisfy i) to iii).
Then, consider the~\mbox{$s_\ell$-$x_\ell$-path} and the predecessor~$x'_\ell$ of~$x_\ell$.
Note that~$x'_\ell$ exists as otherwise~$x_i = s_i$ for all~${i \in [p]}$ and hence we are in the base case.
By construction,~${x'_\ell \in N^*(x_\ell) \subseteq (V \setminus V^{\en}) \cup \{s_i\}}$.
Note further that~$x'_\ell \prec x_\ell \preceq t_\ell$, implying~$x'_\ell \neq t_\ell$ and hence i) is also satisfied by~$x'_\ell$.
Further, since there is an~\mbox{$s_\ell$-$x'_\ell$-path} (a subpath of the~$s_\ell$-$x_\ell$ path), it holds that~${s \preceq x'_\ell \prec x_\ell \preceq t_\ell}$ and thus~$x'_\ell$ also satisfies ii).
Finally, iii) is also satisfied by the~$s_\ell$-$x'_\ell$-subpath combined with the other~$s_i$-$x_i$-paths.
The induction hypothesis then states that \[T[x_1,x_2,\ldots,x_{\ell-1},x'_\ell,x_{\ell+1},\ldots,x_p] = \true.\]
Since~$x'_\ell \in N^*(x_\ell)$, it holds that~$T[x_1,x_2,\ldots,x_p] = \true$.
Thus, the statement holds for all table entries~$T[x_1,x_2,\ldots,x_p]$.

It remains to analyze the running time.
Note that there are at most~$n^p$ possible table entries and computing one entry takes~${O(p + \deg(x_\ell))}$~time as~$V^{\en}, \ell$, and~$N^*(x_\ell)$ can be computed in~${O(p + \deg(x_\ell))}$~time and iterating over all neighbors of~$x_\ell$ takes~$O(\deg(x_\ell))$ time. 
Hence, by the handshaking lemma, the overall running time is in~$O(n^{p-1}m + pn^p)$.
\end{proof}

Making use of \cref{lem:DP}, we can now present our algorithm for solving \kDSP{}.
Pseudo-code is given in \cref{alg:dynProg}.
\begin{algorithm}[t!]
	\caption{Our algorithm for \kDSP.}
	\label{alg:dynProg}
	\myproc{\solve{$G$, $(s_i,t_i)_{i \in [k]}$}}
	{
		\ForEach{possible crossing set~$(T_i)_{i \in [k]}$ with endpoints~$\marbles$ \label{alg:dynProg:guessing}}
		{
			\tcc{We now try to find a solution~$\calP \defeq (P_i)_{i \in [k]}$ with~$T_i = V(P_i) \cap \crossing$ for all~$i \in [k]$ and~$\marbles = \segmentEnds$.}
			\ForEach{$i \in [k]$}{
				$\mathcal P_i \gets \emptyset$ \tcp{$P_i$ contains all segments corresponding to~$D_i$}
				}
			\ForEach{\emph{minimal segment $S$ of some~$T_i$ with~$i \in [k]$\label{alg:dynProg:assignment}}}{
				$\varmarks[S_i] \gets \emptyset$\;
				\ForEach{\emph{permutation~$\sigma = (\ell_1,\ell_2,\ldots,i)$ with~$\marbles(\sigma) = \{\alpha,\omega\} \neq \{\bot\}$ and~$\alpha \leq^i \start(S_i) <^i \en(S_i) \leq^i \omega$}\label{alg:dynProg:marks}}{
					$\varmarks[S_i] \gets \varmarks[S_i] \cup \setp(\sigma)$\;
				}
				$j \gets \min \varmarks[S]$\;
				$x \gets \argmin \{\pos{v}^j \mid v \in \{\start(S),\en(S)\}\}$\label{alg:dynProg:ifs}\;
				$y \gets \argmax \{\pos{v}^j \mid v \in \{\start(S),\en(S)\}\}$\;
				$\mathcal P_j = \mathcal P_j \cup \{ (x,y) \}$ \label{alg:dynProg:ife}\;
			}
			\ForEach{$j \in [k]$}{
					Order~$\mathcal P_j = ((x_1,y_1),(x_2,y_2),\ldots)$ such that~$\pos{x_1}^j \leq \pos{x_2}^j \leq \ldots$ \label{alg:dynProg:sorting}\;
			}
			\lIf{\emph{all instances $(D_i,\mathcal P_i)$ of \textsc{$|\mathcal P_i|$-Disjoint Paths on DAGs} are yes-instances and the combined solutions form a solution of \kDSP}}{
				\KwRet{$\true$} \label{alg:dynProg:yes}
			}
		}
		\KwRet{$\false$}\label{alg:dp-end}
	}
\end{algorithm}%
Next, we show that \cref{alg:dynProg} is correct and runs in~$O(n^{16k+k!+k+1})$ time.
We start with the analysis of the running time.

\begin{lemma}
		\label{lem:running-time}
	\cref{alg:dynProg} runs in~$O(k \cdot n^{16k \cdot k! + k + 1})$ time.
\end{lemma}
\begin{proof}
	First, observe that there are at most~$k \cdot k!$ different permutations of subsets of~$k$ objects as there are exactly~$k!$ permutations of exactly~$k$ objects and each of these can be truncated at~$k$ positions to get any permutation of any smaller (non-empty) subset of objects.
	Second, observe that by \cref{def:crossing-sets} there are at most eight vertices guessed for each sequence~$\sigma$ as if~$\delta^{i,j}_{P}(Q) \neq \bot$, then~$\alpha^{i,j}_{P}(Q) = \omega^{i,j}_{P}(Q) = \partial^{i,j}_{P}(Q)= \varpi^{i,j}_{P}(Q) = \bot$.
	Hence, at most~$8k \cdot k!$ vertices need to be guessed, which requires at most $n^{8k \cdot k!}$~attempts.
		
	Next, we analyze the running time of each iteration of the main foreach-loop in~\cref{alg:dynProg}.
	By \cref{def:crossing-sets}, there are at most four vertices on a \mpath{}~$T_i$ for each sequence~$\sigma$ and each of these vertices increases the number of minimal segments on~$T_i$ by at most one.
	Note that for each~$\sigma$ the set $\crossing^{\sigma}$ contains vertices from at most two paths.
	Thus, we create at most~$8k\cdot k!$ new segments overall.
	Since we start with~$k$ \mpath{s}, there are at most~$8k\cdot k! + k$ minimal segments.
	Thus, there are at most~$(8k\cdot k! + k)\cdot (k\cdot k!)$ iterations of the loop in \cref{alg:dynProg:marks}, each of which takes constant time.
	Since all other operations in the loop in \cref{alg:dynProg:assignment} take constant time, the overall running-time for this loop is in~$O((8k\cdot k! + k)\cdot (k\cdot k!))$.
	Each iteration of \cref{alg:dynProg:sorting} can be done in~$O(n)$ time using bucket sort and hence the overall running time for all iterations is in~$O(n\cdot k)$.

	Next, there are~$k$ instances of \textsc{$p_i$-Disjoint Paths on DAGs} that are solved using \cref{lem:DP}, where~$p_i \leq 8k \cdot k! + k$ and~$p_i \leq n$ for all~$i \in [k]$.
	By \cref{lem:DP}, the running time for solving one instance is in~$O(n^{8k \cdot k! + k-1}m + n^{8k \cdot k! + k + 1}) \subseteq O(n^{8k \cdot k! + k+1})$ and the running time for solving all instances is therefore in~$O(k \cdot n^{8k \cdot k! + k + 1})$.
	Lastly, we verify in \cref{alg:dynProg} that the solutions found can indeed be merged into one solution for \kDSP.
	Note that we only stated the decision version of~\textsc{$p$-Disjoint Paths on DAGs} but the actual solution can be found using a very similar algorithm where we do not only store~$\true$ or~$\false$ in the table~$T$ but also some set of disjoint paths corresponding to each table entry that stores~$\true$.
	Verifying a solution can be done in~$O(k\cdot n)$ time by iterating over all solution paths and verifying that between each pair of consecutive vertices there is an edge, that all paths are shortest paths, and that all paths are internally vertex-disjoint.
	The last verification can be done by marking all inner vertices of each path and if some vertex is already marked once and visited again, then return~$\false$ and otherwise return~$\true$. 
	Thus, the overall running time of \cref{alg:dynProg} is in
	\begin{equation*}
		O(n^{8k \cdot k!} \cdot ((8k\cdot k! + k) \cdot (k\cdot k!) + n \cdot k + k \cdot n^{8k \cdot k! + k + 1} + n\cdot k) \subseteq O(k \cdot n^{16k \cdot k! + k + 1}). \qedhere
	\end{equation*}
\end{proof}

For the correctness of \cref{alg:dynProg}, we need to show that each part of the partition of minimal segments can be solved independently.
This follows from~\cref{lem:correctness} together with the fact that \cref{alg:dynProg} exhaustively tries all possibilities for the crosssing set~$\crossing$.
Together with \cref{lem:running-time}, this implies \cref{thm:main-thm}.

\algorithmThm*

\begin{proof}
	We use \cref{alg:dynProg} and focus on the correctness as the running time is already analyzed in \cref{lem:running-time}.
	If \cref{alg:dynProg} returns~$\true$, then \cref{alg:dynProg:yes} is executed and a solution is verified.
	It remains to show that if there is some solution, then \cref{alg:dynProg} returns~$\true$.
	If there is some solution~$\calP = (P_i)_{i \in [k]}$, then let~$\crossing$ be its crossing set (\cref{def:crossing-sets}).
	Then, there is some iteration of \cref{alg:dynProg:guessing} where all guesses are correct, that is, $\marbles = \segmentEnds$ and~${T_i = V(P_i) \cap \crossing}$.
	We now consider this iteration of \cref{alg:dynProg:guessing}.
	
	Note that \cref{thm:crossing-set-valid} states that the pair~$\{\start(S),\en(S)\}$ is~$c$-colored for each sequence~$\sigma$, each segment~$S$ with~$\{\start(S),\en(S)\} = \marbles(\sigma) = \segmentEnds(\sigma)$ and for each~$c \in \setp(\sigma)$.
	Hence, the same also holds for each minimal segment~$S' \subseteq S$.
	By \cref{alg:dynProg:marks}, there is a solution where the shortest paths between the endpoints of each minimal segment~$S$ are strictly~$c$-monotone for each~${c \in \varlabels[S]}$.
	Note that~${\varlabels[S] = \varmarks[S]}$ in this iteration of \cref{alg:dynProg:guessing}.
	Hence, each path following~$S$ is strictly~$c$-increasing for each~${c \in \varmarks[S]}$ and by \cref{obs:path-dag} this shortest path is contained in~$D_c$.
	Thus, we can find \emph{some} solution for each minimal segment using \cref{lem:DP} such that all paths for these minimal segments with the same marks are internally vertex-disjoint.
	Since~$\varmarks[S] = \varlabels[S]$ for all minimal segments, by \cref{lem:correctness}, all shortest paths between endpoints of minimal segments with different marks are internally vertex-disjoint.
	Hence, the result computed by \cref{alg:dynProg} is a solution to \kDSP{} and thus the algorithm returns~$\true$.
\end{proof}

\section{Lower bounds for \kDSP} %
\label{sec:eth}
The \emph{Exponential-Time Hypothesis (ETH)} states that 
there is no $2^{o(n)}$-time algorithm for \textsc{3-SAT}, where $n$ is the number of variables \cite{impagliazzo2001complexity}.
We show that there is no $f(k)\cdot n^{o(k)}$-time algorithm for \kDSP, unless the ETH fails.

\hardnessThm*
\begin{proof}
	We reduce from \textsc{Multicolored Clique} which is defined as follows. 
	\decProb{Multicolored Clique (MCC)}
	{A graph $G=(V,E)$, an integer~$k\in \NN$, and a coloring function~$c\colon V \rightarrow [k]$. }
	{Is there a multicolored clique in~$G$, that is, a set of pairwise adjacent vertices containing exactly one vertex of each color in~$[k]$?}

	We provide a polynomial-time reduction 
	from an instance~$(G,c,k)$ of \textsc{Multicolored Clique} to an instance~$(G' = (V',E'),(s_i,t_i)_{i\in[2k]})$ of \textsc{$2k$-DSP}.
	
	The basic idea is as follows: 
	For each vertex in~$G$ we will have a ``horizontal'' and a ``vertical'' path in~$G'$.
	These paths will form a grid in~$G'$, see \cref{fig:hardness-example} for an illustration. 
	\begin{figure}[t]
		\newcommand{\edgeInKDSPInstance}{
			\node[vertex,fill=white,minimum size=5pt] at (\i + \j * \smallDist - 2 * \smallDist,2 * \smallDist + -\x - \y * \smallDist) {};
			\node[vertex,fill=white,minimum size=5pt] at (\x + \y * \smallDist - 2 * \smallDist,2 * \smallDist + -\i - \j * \smallDist) {};
			\node[vertex,fill=black] at (\i + \j * \smallDist - 2 * \smallDist - \smallerDist,2 * \smallDist + -\x - \y * \smallDist) {};
			\node[vertex,fill=black] at (\i + \j * \smallDist - 2 * \smallDist,2 * \smallDist + -\x - \y * \smallDist + \smallerDist) {};
			\node[vertex,fill=black] at (\x + \y * \smallDist - 2 * \smallDist - \smallerDist,2 * \smallDist + -\i - \j * \smallDist) {};
			\node[vertex,fill=black] at (\x + \y * \smallDist - 2 * \smallDist,2 * \smallDist + -\i - \j * \smallDist + \smallerDist) {};
		}
		\centering
		\begin{tikzpicture}
			\newcommand{\colorA}{blue}
			\newcommand{\colorB}{red!90!black}
			\newcommand{\colorC}{green!60!black}
			\newcommand{\colorD}{black}
			
			\newcommand{\smallDist}{0.3}
			\newcommand{\smallerDist}{0.1}
			\newcommand{\cliqueVertices}{1, 3, 2, 1} %
		
			\begin{scope}[yshift=-1cm]
				\foreach \x in {1,...,3}{
					\node[vertex,fill=\colorA,
						label=above:{$1,\x$}
						] (v-1\x) at (\x,0) {};
					\node[vertex,fill=\colorB,
						label=right:{$2,\x$}
						] (v-2\x) at (4,-\x) {};
					\node[vertex,fill=\colorC,
						label=below:{$3,\x$}
						] (v-3\x) at (\x,-4) {};
					\node[vertex,fill=\colorD,
						label=left:{$4,\x$}
						] (v-4\x) at (0,-\x) {};
				}
			\end{scope}
			
			\begin{scope}[xshift=5.5cm, xscale=1.4, yscale=1.2]
			
				\foreach[count=\i] \color in {\colorA,\colorB,\colorC,\colorD}{
					\node[vertex,fill=\color,label=below:{$s_\i$}] (s-\i) at (0,-\i) {};
					\node[vertex,fill=\color,label=below:{$t_\i$}] (t-\i) at (5,-\i) {};
					
					\pgfmathtruncatemacro\j{\i + 4};
					
					\node[vertex,fill=\color,label=right:{$t_\j$}] (tt-\i) at (\i,-5) {};
					\node[vertex,fill=\color,label=right:{$s_\j$}] (ss-\i) at (\i,0) {};
					\foreach \j in {1,...,3}
					{
						\node[vertex,fill=\color] (sc-\j\i) at (\smallDist,2 * \smallDist + -\i - \j * \smallDist) {};
						\draw[dotted,thick] (sc-\j\i) -- (s-\i);
						\node[vertex,fill=\color] (tc-\j\i) at (4.7,2 * \smallDist + -\i - \j * \smallDist) {};
						\draw[dotted,thick] (tc-\j\i) -- (t-\i);
						\draw (tc-\j\i) -- (sc-\j\i);

						\node[vertex,fill=\color] (ssc-\j\i) at (\i + \j * \smallDist - 2 * \smallDist,-\smallDist) {};
						\draw[dotted,thick] (ssc-\j\i) -- (ss-\i);
						\node[vertex,fill=\color] (ttc-\j\i) at (\i + \j * \smallDist - 2 * \smallDist,-4.7) {};
						\draw[dotted,thick] (ttc-\j\i) -- (tt-\i);
						\draw (ttc-\j\i) -- (ssc-\j\i);
					}
				}

				\begin{pgfonlayer}{background}
					\foreach \i in {1,...,4} \foreach \x in {1,...,4} \foreach \j in {1,...,3} \foreach \y in {1,...,3}
					{
						\node[vertex,fill=black] at (\i + \j * \smallDist - 2 * \smallDist,2 * \smallDist + -\x - \y * \smallDist) {};
						\node[vertex,fill=black] at (\x + \y * \smallDist - 2 * \smallDist,2 * \smallDist + -\i - \j * \smallDist) {};
					}

					\foreach \i in {1,...,4} \foreach \j in {1,...,3} \foreach \y in {1,...,3}
					{
						\ifnum \j=\y%
							\pgfmathtruncatemacro\x{\i};
							\edgeInKDSPInstance
						\fi%
					}

					\foreach \i / \j / \x / \y in {
							1/3/3/3, 1/2/3/3, 3/1/4/3, 3/3/2/1, 1/3/4/1, 2/2/1/3, 1/1/4/3, 3/3/4/2, 1/2/2/2, 1/2/3/1, 2/1/4/3, 2/2/4/2, 2/3/3/1, 4/2/3/2%
						} 
					{
						\draw (v-\i\j) -- (v-\x\y);
						\edgeInKDSPInstance
					}
					\foreach[count=\i] \j in \cliqueVertices
					{
						\foreach[count=\x] \y in \cliqueVertices
						{
							\ifnum \i<\x%
								\draw[very thick] (v-\i\j) -- (v-\x\y);
								\edgeInKDSPInstance
							\fi
						}
						\draw[path] (s-\i.center) -- (sc-\j\i.center) -- (tc-\j\i.center) -- (t-\i.center);
						\draw[path] (ss-\i.center) -- (ssc-\j\i.center) -- (ttc-\j\i.center) -- (tt-\i.center);
					}
				\end{pgfonlayer}
				
			\end{scope}
			
		\end{tikzpicture}
		\caption{
			An illustration of the reduction from \textsc{Multicolored Clique} to \kDSP. \\
			\emph{Left side:} Example instance for \textsc{Multicolored Clique} with~$k=4$ colors and three vertices per color.
			A multicolored clique is highlighted (by thick edges). \\
			\emph{Right side:} The constructed instance with the eight shortest paths highlighted.
			Note that these paths are pairwise disjoint.
			Dashed edges (incident to~$s_i$ and~$t_i$ vertices) indicate paths of length~12.
			This ensures that no shortest~$s_i$-$t_i$-path contains a vertex~$s_j$ or~$t_j$ with~$i \ne j$.
		}
		\label{fig:hardness-example}
	\end{figure}
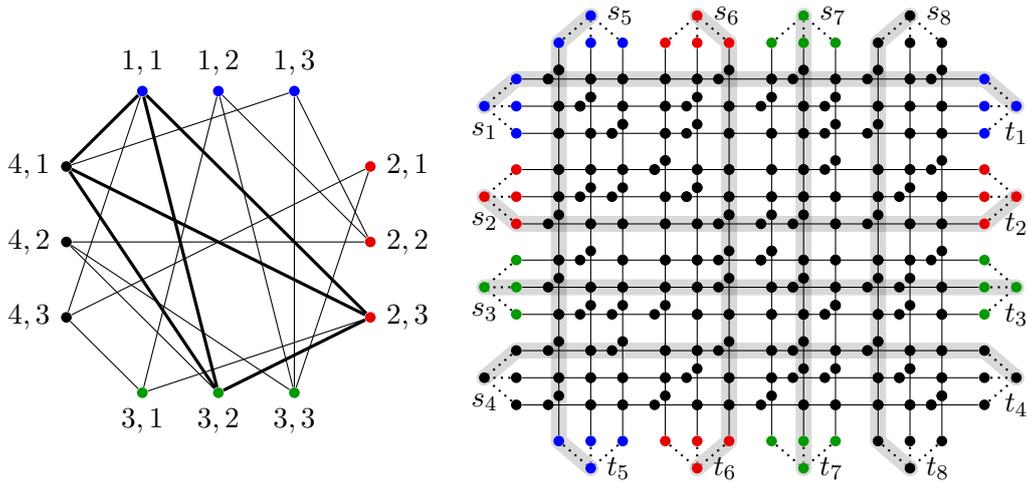
	With the pairs~$(s_a,t_a)$ and~$(s_{a+k},t_{a+k}), a \in [k]$, we ensure that for each color~$a$ a horizontal and a vertical path corresponding to a vertex of color~$a$ has to be taken into a solution. 
	If two vertices~$u$ and~$v$ cannot appear together in a clique (because they have the same color or are not adjacent), then the horizontal $u$-path and the vertical~$v$-path will intersect (so not both paths can be taken at the same time). 
	Note that the highlighted paths in the right side of \cref{fig:hardness-example} do not cross at vertices.
	The shortest path property ensures that a $s_a$-$t_a$-path cannot ``switch'' between horizontal or vertical paths in the grid.
	If there are~$2k$ disjoint shortest paths connecting~$s_a$ and~$t_a$, then the horizontal and vertical paths will not cross. 
	Thus, the corresponding vertices will form a clique in~$G$. 
	
	The details of the construction are as follows: 
	Initialize~$G'$ as the empty graph.
	For each color~$a \in [k]$ we add four vertices~$s_a$, $t_a$, $s_{k+a}$, and $t_{k+a}$ to~$V'$.
	For each vertex~$v \in V$, we add two paths~$P_v = s_{c(v)} p_v^1 p_v^2 \dots p_v^n t_{c(v)}$ and~$Q_v = s_{k+c(v)} q_v^1 q_v^2 \dots q_v^n t_{k+c(v)}$ where the~$p_v^i$ and~$q_v^i$ are new vertices.
	Afterwards, subdivide the first and last edge of each of these paths by inserting~$n$~new degree-2 vertices (indicated by the dotted lines in \cref{fig:hardness-example}).
	
	To finish the construction, we merge some vertices.
	Here, merging two vertices~$u$ and~$v$ means to add a new vertex~$w$ with~$N(w) = N(u) \cup N(v)$ and to delete~$u$ and~$v$ afterwards.
	To simplify notation, we use both names~$u$ and~$v$ for the new merged vertex~$w$.
	Fix any vertex order $V = \{v_1, \dots, v_n\}$.
	For each pair~$v_i, v_j$ with $i \neq j$, we merge~$p_{v_i}^j$ and $q_{v_j}^i$ if 
	\begin{align}
		&c(v_i) = c(v_j) \quad\text{or} \label{cond:same-color}
		\\
		&\{v_i,v_j\} \notin E. \label{cond:different-color}
	\end{align}

	It remains to show the correctness of the construction.
	We begin by observing that the following two inequalities hold for all $i, j, a$,
	as they are valid before the vertex merging	and are kept invariant by each merge step.
	\begin{align*}   
		\dist(s_a, p_{v_i}^j) &\geq n + j  &
		\dist(s_a, q_{v_i}^j) &\geq n + i
	\end{align*}
	For any vertex $v$ and $j \in [n]$, the existence of path $P_v$ proves $\dist(s_{c(v)}, p_v^j) \leq n+j$
	and thus~$\dist(s_{c(v)}, p_v^j) = n+j$.
	
	Observe that $p_{v_i}^j$ is only merged with $q_{v_j}^i$ (if at all).
	Hence, $N(p_{v_i}^j) \subseteq \set{p_{v_i}^{j-1}, p_{v_i}^{j+1}, q_{v_j}^{i-1}, q_{v_j}^{i+1}}$.
	Of these four vertices only $p_{v_i}^{j-1}$ has $\dist(s_a, p_{v_i}^{j-1}) \leq n+j-1$.
	It follows inductively that the only path of length $n+j$ connecting $p_{v_i}^j$ to any $s_a$ is a subpath of $P_{v_i}$.
	Thus, the set of shortest~\mbox{$s_{a}$-$t_{a}$-paths} in $G'$
	is exactly~$\set{P_{v}; c(v) = a}$.
	An analogous argument implies that~${\set{Q_v; c(v) = a}}$ is the set of shortest $s_{k+a}$-$t_{k+a}$-paths.
	
	Using this, we can now prove that~$G$ contains a multi-colored clique~$C \subseteq V$ of size at most~$k$ if and only if there are~$2k$ shortest and pairwise disjoint paths in~$G'$ that connect~$(s_i,t_i)_{i\in[2k]}$.
	
	``$\Rightarrow$:''
	Given the multicolored clique~$C$, we select the following paths in~$G'$: 
	For each vertex~$v \in C$, we take the two paths~$P_v$ and $Q_v$.
	As shown above these are shortest paths.
	It remains to show that they are pairwise disjoint: 
	By construction, any two selected paths~$P_{v_i}, Q_{v_j}$ overlap in some vertex if and only if Condition~\eqref{cond:same-color} or Condition~\eqref{cond:different-color} is satisfied.
	Since~$v_i, v_j \in C$, it follows that~$\{v_i,v_j\} \in E$ for~$c(v_i) \neq c(v_j)$. 
	Moreover, for~$c(v_i) = c(v_j)$, we have~$v_i = v_j$ since we take exactly one vertex per color into the clique.
	Thus, the paths are pairwise disjoint.
	
	``$\Leftarrow$:''
	Assume that $R_a$ is a shortest $s_a$-$t_a$-path for each $a \in [2k]$ and that these paths are pairwise disjoint.
	As previously shown, for each $a$ we have $R_a = P_v$ (if $a \leq k$) or~$R_a = Q_v$ (if~$a > k$) for some vertex $v$ with $c(v) = a$.
	Further observe that for any $a \in [k]$,~$R_a$ and~$R_{k + a}$ correspond to the same vertex~$v$, i.e., $R_a = P_v$ and $R_{k+a} = Q_v$,
	since otherwise Condition~\ref{cond:same-color} would imply that these two paths intersect.
	Hence, the~$2k$ paths correspond to~$k$ vertices of~$G$ and, due to Condition~\ref{cond:different-color}, it follows that these vertices form a clique in $G$.
	
	By observing that the provided reduction is a parameterized reduction and the fact that \textsc{Multicolored Clique} is W[1]-hard~\cite{CFK+15}, we obtain the first part of the theorem, that is, the W[1]-hardness of \kDSP{} with respect to~$k$.
	The second part of the theorem follows as there is no $f(k)\cdot n^{o(k)}$-time algorithm for MCC~\cite{chen2006strong} (unless the ETH fails) and the number of terminal pairs in the \textsc{$2k$-DSP} instance depends linearly on $k$.
\end{proof}

Note that all edges in the previous reduction are only ever used ``from~$s_i$ to~$t_i$''.
Thus, we can easily modify the reduction by directing all edges in this way to also show W[1]-hardness with respect to~$k$ for the directed variant.
Moreover, by replacing each vertex~$u$ by two vertices~$u^{\iin}$ and~$u^{\oout}$ where~$u^{\iin}$ is incident to all incoming arcs and~$u^{\oout}$ is incident to all outgoing arcs of~$u$ (and adding the arc~$(u^{\iin},u^{\oout})$), we also get the same hardness for the variant where we ask for edge-disjoint~\mbox{$s_i$-$t_i$-paths} instead of vertex disjoint paths.
\section{Conclusion}

We provided an improved algorithm for \kDSP.
However, while the running time of our algorithm can certainly be slightly improved by some case distinctions and a more careful analysis, the algorithm is still far from being practical.
Reducing the factor in the exponent to a polynomial in~$k$ is a clear challenge for future work.
Considering the fine-grained complexity of \DSP{2}, it would be interesting to know whether there are running time barriers based on e.\,g.\;the Strong Exponential Time Hypothesis. 

Concerning generalizations of \kDSP, we believe that we can modify our algorithm in a straight-forward way to work with positive edge-lengths.
However, the case of non-negative edge-lengths seems much more difficult.
Our basic geometric observations made in \cref{sec:2D-geometry} crucially depend on the fact that we are looking for shortest paths.
Thus, if there are no~$k$~disjoint shortest paths, then computing $k$~disjoint paths minimizing their total length in polynomial time is still an open problem for $k\geq 3$
(for~$k=2$ \citet{BH19} provided a randomized~$O(n^{11})$ time algorithm).

\bibliographystyle{plainnat}

\bibliography{bib}

\newcommand{\bibremark}[1]{}
\begin{thebibliography}{16}
\providecommand{\natexlab}[1]{#1}
\providecommand{\url}[1]{\texttt{#1}}
\expandafter\ifx\csname urlstyle\endcsname\relax
  \providecommand{\doi}[1]{doi: #1}\else
  \providecommand{\doi}{doi: \begingroup \urlstyle{rm}\Url}\fi

\bibitem[Akhmedov(2020)]{Akh20}
Maxim Akhmedov.
\newblock Faster 2-disjoint-shortest-paths algorithm.
\newblock In \emph{Proceedings of the 15th International Computer Science
  Symposium in Russia ({CSR}~'20)}, volume 12159 of \emph{Lecture Notes in
  Computer Science}, pages 103--116. Springer, 2020.
\newblock \doi{10.1007/978-3-030-50026-9\_7}.

\bibitem[Berczi and Kobayashi(2017)]{BK17}
Kristof Berczi and Yusuke Kobayashi.
\newblock {The Directed Disjoint Shortest Paths Problem}.
\newblock In \emph{Proceedings of the 25th Annual European Symposium on
  Algorithms (ESA~'17)}, volume~87 of \emph{LIPIcs}, pages 13:1--13:13. Schloss
  Dagstuhl--Leibniz-Zentrum für Informatik, 2017.
\newblock \doi{10.4230/LIPIcs.ESA.2017.13}.

\bibitem[Bj{\"{o}}rklund and Husfeldt(2019)]{BH19}
Andreas Bj{\"{o}}rklund and Thore Husfeldt.
\newblock Shortest two disjoint paths in polynomial time.
\newblock \emph{SIAM Journal on Computing}, 48\penalty0 (6):\penalty0
  1698--1710, 2019.
\newblock \doi{10.1137/18M1223034}.

\bibitem[Chen et~al.(2006)Chen, Huang, Kanj, and Xia]{chen2006strong}
Jianer Chen, Xiuzhen Huang, Iyad~A Kanj, and Ge~Xia.
\newblock Strong computational lower bounds via parameterized complexity.
\newblock \emph{Journal of Computer and System Sciences}, 72\penalty0
  (8):\penalty0 1346--1367, 2006.
\newblock \doi{10.1016/j.jcss.2006.04.007}.

\bibitem[Cygan et~al.(2015)Cygan, Fomin, Kowalik, Lokshtanov, Marx, Pilipczuk,
  Pilipczuk, and Saurabh]{CFK+15}
Marek Cygan, Fedor~V. Fomin, Lukasz Kowalik, Daniel Lokshtanov, D{\'{a}}niel
  Marx, Marcin Pilipczuk, Michal Pilipczuk, and Saket Saurabh.
\newblock \emph{Parameterized Algorithms}.
\newblock Springer, 2015.
\newblock \doi{10.1007/978-3-319-21275-3}.

\bibitem[Eilam{-}Tzoreff(1998)]{Eil98}
Tali Eilam{-}Tzoreff.
\newblock The disjoint shortest paths problem.
\newblock \emph{Discrete Applied Mathematics}, 85\penalty0 (2):\penalty0
  113--138, 1998.
\newblock \doi{10.1016/S0166-218X(97)00121-2}.

\bibitem[Fomin et~al.(2019)Fomin, Marx, Saurabh, and Zehavi]{FMSZ19}
Fedor~V. Fomin, D{\'a}niel Marx, Saket Saurabh, and Meirav Zehavi.
\newblock {New Horizons in Parameterized Complexity (Dagstuhl Seminar 19041)}.
\newblock \emph{Dagstuhl Reports}, 9\penalty0 (1):\penalty0 67--87, 2019.
\newblock \doi{10.4230/DagRep.9.1.67}.

\bibitem[Fortune et~al.(1980)Fortune, Hopcroft, and Wyllie]{FHW80}
Steven Fortune, John~E. Hopcroft, and James Wyllie.
\newblock The directed subgraph homeomorphism problem.
\newblock \emph{Theoretical Computer Science}, 10:\penalty0 111--121, 1980.
\newblock \doi{10.1016/0304-3975(80)90009-2}.

\bibitem[Gottschau et~al.(2019)Gottschau, Kaiser, and Waldmann]{GKW19}
Marinus Gottschau, Marcus Kaiser, and Clara Waldmann.
\newblock The undirected two disjoint shortest paths problem.
\newblock \emph{Operations Research Letters}, 47\penalty0 (1):\penalty0 70--75,
  2019.
\newblock \doi{10.1016/j.orl.2018.11.011}.

\bibitem[Impagliazzo and Paturi(2001)]{impagliazzo2001complexity}
Russell Impagliazzo and Ramamohan Paturi.
\newblock On the complexity of k-sat.
\newblock \emph{Journal of Computer and System Sciences}, 62\penalty0
  (2):\penalty0 367--375, 2001.
\newblock \doi{10.1006/jcss.2000.1727}.

\bibitem[Karp(1975)]{Kar75}
Richard~M. Karp.
\newblock On the computational complexity of combinatorial problems.
\newblock \emph{Networks}, 5\penalty0 (1):\penalty0 45--68, 1975.
\newblock \doi{10.1002/net.1975.5.1.45}.

\bibitem[Kawarabayashi et~al.(2012)Kawarabayashi, Kobayashi, and Reed]{KKR12}
Ken{-}ichi Kawarabayashi, Yusuke Kobayashi, and Bruce~A. Reed.
\newblock The disjoint paths problem in quadratic time.
\newblock \emph{Journal of Combinatorial Theory. Series~B}, 102\penalty0
  (2):\penalty0 424--435, 2012.
\newblock \doi{10.1016/j.jctb.2011.07.004}.

\bibitem[Kobayashi and Sako(2019)]{KS19}
Yusuke Kobayashi and Ryo Sako.
\newblock Two disjoint shortest paths problem with non-negative edge length.
\newblock \emph{Operations Research Letters}, 47\penalty0 (1):\penalty0 66--69,
  2019.
\newblock \doi{10.1016/j.orl.2018.11.012}.

\bibitem[Lochet(2021)]{Loc21}
William Lochet.
\newblock A polynomial time algorithm for the $k$-disjoint shortest paths
  problem.
\newblock In \emph{Proceedings of the 32nd ACM-SIAM Symposium on Discrete
  Algorithms (SODA~'21)}, pages 169--178. SIAM, 2021.
\newblock \doi{10.1137/1.9781611976465.12}.

\bibitem[Robertson and Seymour(1995)]{RS95}
Neil Robertson and Paul~D. Seymour.
\newblock Graph minors. {XIII}: the disjoint paths problem.
\newblock \emph{Journal of Combinatorial Theory. Series~B}, 63\penalty0
  (1):\penalty0 65--110, 1995.
\newblock \doi{10.1006/jctb.1995.1006}.

\bibitem[Tholey(2012)]{Tho12}
Torsten Tholey.
\newblock Linear time algorithms for two disjoint paths problems on directed
  acyclic graphs.
\newblock \emph{Theoretical Compututer Science}, 465:\penalty0 35--48, 2012.
\newblock \doi{10.1016/j.tcs.2012.09.025}.

\end{thebibliography}

\end{document}